\documentclass[12pt]{amsart}
\usepackage[hmargin=1in,vmargin=1in]{geometry}
\usepackage{url,  mathrsfs}
\usepackage{hyperref}
\usepackage{fancyhdr} 
\usepackage[T1]{fontenc}
\usepackage{url,  mathrsfs}
\usepackage{hyperref}
\usepackage{amsmath}
\usepackage{graphicx,  epstopdf}
\usepackage{subfig}
\usepackage{color}
\usepackage{bbm}
\usepackage{subfloat}
\usepackage[draft]{fixme}
\usepackage{bm} 
\usepackage{bbm} 
\usepackage{epigraph}
\usepackage{endnotes}
\usepackage{ifpdf}
\usepackage{array}

\usepackage{multicol}
\usepackage{multirow}
\usepackage{graphicx}
\usepackage{tikz}
\usetikzlibrary{positioning} 
\def\dashmapsto{\mapstochar\dashrightarrow}

\newcommand{\PP}{\mathbb{P}}

\newcommand{\N}{\mathbb{N}}
\newcommand{\Z}{\mathbb{Z}}
\newcommand{\R}{\mathbb{R}}

\newcommand{\C}{\mathbb{C}}

\newcommand{\V}{\mathcal{V}}
\newcommand{\I}{\mathcal{I}}
\newcommand{\J}{\mathcal{J}}
\newcommand{\x}{\mathbf{x}}
\newcommand{\p}{\mathbf{p}}
\newcommand{\cL}{\mathcal{L}}
\newcommand{\B}{\mathcal{B}}
\newcommand{\CC}{\mathcal{C}}
\newcommand{\cP}{\mathcal{P}}

\newcommand{\ol}{\overline}
\newcommand{\ii}{\textnormal{i}}

\DeclareMathOperator{\In}{in}

\DeclareMathOperator{\inv}{inv}
\DeclareMathOperator{\supp}{supp}

\newtheorem{thm}{Theorem}[section]
\newtheorem{lem}[thm]{Lemma}
\newtheorem{cor}[thm]{Corollary}
\newtheorem{prop}[thm]{Proposition}

\theoremstyle{definition}
\newtheorem{definition}[thm]{Definition}

\newtheorem{ex}[thm]{Example}
\newtheorem{rem}[thm]{Remark}

\usepackage{amssymb}

\title{Semi-inverted linear spaces and an analogue of the broken circuit complex}
\author{Georgy Scholten}
\address{North Carolina State University, Deptartment of Mathematics,Raleigh, NC 27695, USA}
\email{ghscholt@ncsu.edu}

\author{Cynthia Vinzant}
\address{North Carolina State University, Deptartment of Mathematics,Raleigh, NC 27695, USA}
\email{clvinzan@ncsu.edu}

\date{\today}

\begin{document}

\begin{abstract}
The image of a linear space under inversion of some coordinates is an affine variety whose structure is governed by an underlying 
hyperplane arrangement. In this paper, we generalize work by Proudfoot and Speyer to show that circuit polynomials form a 
universal Gr\"obner basis for the ideal of polynomials vanishing on this variety. The proof relies on degenerations to the Stanley-Reisner ideal of 
a simplicial complex determined by the underlying matroid, which is closely related to the external activity complex defined by Ardila and Boocher. 
If the linear space is real, then the semi-inverted linear space 
is also an example of a hyperbolic variety, meaning that all of its intersection points with a large family of linear spaces are real. 
\end{abstract}

\maketitle

\section{Introduction}
In 2006, Proudfoot and Speyer showed that the coordinate ring of a reciprocal linear space (i.e. the closure of the image of 
a linear space under coordinate-wise inversion) has a flat degeneration into the Stanley-Reisner ring of the broken circuit complex of a matroid \cite{PS}. 
This completely characterizes the combinatorial data of these important varieties, which appear across many areas of mathematics, including 
in the study of matroids and hyperplane arrangements \cite{Terao}, interior point methods for linear programming \cite{DLSV12}, and entropy maximization for log-linear models in statistics \cite{MSUZ14}.  

In this paper we extend the results of Proudfoot and Speyer to the image of a linear space $\cL\subset \C^n$ under inversion of some subset of coordinates.
For $I\subseteq \{1,\hdots, n\}$, consider the rational map  $\inv_I: \C^n \dashrightarrow  \C^n $ defined by 
\[(\inv_I(x))_i =
\begin{cases}
1/x_i  &\text{ if } i\in I\\
 \ \ x_i  &\text{ if } i\not\in I.\\
\end{cases}
\]
Let $\inv_I(\cL)$ denote the Zariski-closure of the image of $\cL$ under this map, which is an affine variety in $\C^n$. One can interpret $\inv_I(\cL)$ as 
an affine chart of the closure of $\cL$ in the product of projective spaces $(\PP^1)^n$, as studied  in \cite{AB16}, or as the projection 
of the graph of $\cL$ under the map $x\dashmapsto \inv_{[n]}(x)$, studied in  \cite{FSW17}, onto complementary subsets of the $2n$ coordinates. 
We give a degeneration of the coordinate ring of $\inv_I(\cL)$
to the Stanley-Reisner ring of a simplicial complex generalizing the broken circuit complex of a matroid. 
This involves constructing a universal Gr\"obner basis for the ideal of polynomials vanishing on $\inv_I(\cL)$.

Let $\C[\x]$ denote the polynomial ring $\C[x_1, \hdots ,x_n]$ and for any $\alpha \in (\Z_{\geq 0})^n$, let 
$\x^{\alpha}$ denote $\prod_{i=1}^n x_i^{\alpha_i}$.  For a subset $S\subseteq[n]$, we will also use $\x^S$ to denote $\prod_{i\in S}x_i$. 
As in \cite{PS}, the \emph{circuits} of the matroid $M(\cL)$ corresponding to $\cL$ give rise to a universal Gr\"obner basis for the 
ideal of polynomials vanishing on $\inv_I(\cL)$. 
We say that a linear form $\ell(x)= \sum_{i\in [n]} a_i x_i$ vanishes on $\cL$ if $\ell(x) = 0$ for all $x\in \cL$.  
The support of $\ell$, $\supp(\ell)$, is $\{i\in [n] : a_i\neq 0\}$. The minimal supports of nonzero linear forms vanishing on $\cL$
are called \emph{circuits} of the matroid $M(\cL)$ and for every circuit $C\subset [n]$, there is a unique (up to scaling) 
linear form $\ell_C= \sum_{i\in C} a_i x_i$ vanishing on $\cL$ with support $C$. To each circuit, we associate the polynomial 
\begin{equation}\label{eq:fC_def}
f_C(\x) \ = \  \x^{C\cap I} \cdot \ell_C(\inv_I(\x)) \ = \ \sum_{i\in C\cap I} a_i \x^{C\cap I\backslash\{ i\}}  +  \sum_{i\in C\backslash I } a_i \x^{C\cap I \cup\{ i\}}.
\end{equation}

\begin{thm}\label{thm:main}
Let $\cL\subseteq \C^n$ be a $d$-dimensional linear space and 
let $\I \subseteq \C[\x]$ be the ideal of polynomials vanishing on $\inv_I(\cL)$. 
Then  $\{f_C : C \text{ is a circuit of }M(\cL)\}$ is a universal Gr\"obner basis for $\I$. 
For $w\in (\R_+)^n$ with distinct coordinates, the initial ideal $\In_w(\I)$ is the Stanley-Reisner 
ideal of the semi-broken circuit complex $\Delta_w(M(\cL), I)$. 
\end{thm}

The simplicial complex $\Delta_w(M(\cL), I)$ will be defined in Section~\ref{sec:SimplicialComplex}.
For real linear spaces $\cL$, the variety $\inv_I(\cL)$ relates to the regions of a hyperplane arrangement.

\begin{thm}\label{thm:real}
Let $\cL\subseteq \C^n$ be a linear space that is invariant under complex conjugation. Then the following numbers are equal:
\begin{enumerate}
\item the degree of the affine variety $\inv_I(\cL)$,
\item the number of facets of the  semi-broken circuit complex $\Delta_w(M(\cL), I)$, and 
\item for generic $u\in \R^n$, the number of
regions in $(\cL^{\perp}+u)\backslash \{x_i = 0\}_{i\in I}$ 
whose recession cones trivially intersect $\R^I = \{x \in \R^n : x_j = 0 \text{ for } j\not\in I\}$. 
\end{enumerate}
\end{thm}

The paper is organized as follows. The necessary definitions and background on matroid theory and Stanley-Reisner ideals 
are in Section~\ref{Background}. In Section~\ref{sec:SimplicialComplex} we define the simplicial complex $\Delta_w(M(\cL), I)$,
show that it satisfies a deletion-contraction relation analogous to that of the broken circuit complex of a matroid, and describe its relationship
to the external activity complex of a matroid.  Section~\ref{sec:proof} 
contains the proof of Theorem~\ref{thm:main}. We characterize the strata of $\inv_I(\cL)$ given by its intersection with coordinate subspaces in Section~\ref{sec:Supports}.
Finally, in Section~\ref{sec:real}, we show that for a real linear space $\cL$, $\inv_I(\cL)$ is a \emph{hyperbolic variety}, in the sense of \cite{KV,SV}, 
and prove Theorem~\ref{thm:real}. 

{\bf Acknowledgements.} The authors would like to thank Nicholas Proudfoot, Seth Sullivant, and Levent Tun\c cel for useful discussions over the course of this project and the anonymous referees for their careful reading and helpful suggestions.

\section{Background}\label{Background}

In this section, we review the necessary background on Gr\"obner bases, simplicial complexes, Stanley-Reisner ideals, matroids, and 
previous research on reciprocal linear spaces. 

\subsection{Gr\"obner bases and degenerations} \label{subsec:Grobner}

A finite subset $F$ of an ideal $\I \subset \C[\x]$ is a {\bf universal Gr\"obner basis} for $\I$ 
if it is a Gr\"obner basis with respect to every monomial order on $\C[\x]$. 
An equivalent definition using weight vectors is given as follows. For $w\in (\R_{\geq 0})^n$ and $f = \sum_\alpha c_{\alpha} \x^{\alpha} \in \C[\x]$, 
define the degree and initial form of $f$ with respect to $w$ to be 
\[
\deg_w(f) = \max\{w^T\alpha : c_{\alpha}\neq 0\} \ \ \text{ and } \ \ 
\In_w(f) \ = \ \sum_{\alpha: w^T\alpha = \deg_w(f)  } c_{\alpha} \x^{\alpha}. \]
The initial ideal $\In_w(\I)$ of an ideal $\I$ is the ideal generated by initial forms of polynomials in $\I$, i.e. $\In_w(\I) = \langle \In_w(f) : f\in \I\rangle$. 
Then $F\subset \I$ is a universal Gr\"obner basis for $\I$ if and only if for every $w\in (\R_{\geq 0})^n$, the polynomials $\In_w(F)$ generate $\In_w(\I)$. i.e.
See \cite[Chapter 1]{GrobnerPolytope}. 

For homogeneous $\I$, $\In_w(\I)$ is a \emph{flat degeneration} of $\I$. 
For $f\in \C[\x]$ and an integer vector $w\in (\Z_{\geq 0})^n$, define 
\[t^w\cdot f = t^{\deg_w(f)}f(t^{-w_1}x_1, \hdots, t^{-w_n}x_n) \in \C[t,\x]
\ \ \text{ and } \ \
\ol{\I}^w = \langle t^{w}\cdot f : f\in \I \rangle \subset \C[t,\x].
\]
The ideal $\ol{\I}^w$ defines a variety in $\mathbb{A}^1(\C) \times \PP^{n-1}(\C)$,
namely the Zariski-closure 
\[
\V(\ol{\I}^w) = \ol{\bigl\{ (t, [t^{w_1} x_1 : \hdots: t^{w_n}x_n ]) \ \text{ such that }  \ t\in {\C^*}, x\in \V(\I)\bigl\}}^{\rm Zar}.
\]
Letting $t$ vary from $1$ to $0$ gives a flat deformation from $\V(\I)$ to 
the variety $\V(\In_w(\I))$. Formally, for any $\gamma\in \C$, let $\ol{\I}^w(\gamma)$ denote the ideal in $\C[\x]$
obtained by substituting $t = \gamma$. 
Then $\ol{\I}^w(1)$ equals $ \I$,  $\ol{\I}^w(0)$ equals $\In_w(\I)$, 
and for $\gamma\in \C^*$, the variety of $\ol{\I}^w(\gamma)$ 
consists of the points $\{[\gamma^{w_1} x_1 : \hdots : \gamma^{w_n}x_n] : x\in \V(\I)\}$.  
We note that the ideal $\ol{\I}^w(\gamma)$ is well-defined for any $w\in \R^n$. 
All the ideals $\ol{\I}^w(\gamma)$ have the same Hilbert series. 
In particular, taking $\gamma = 0,1$ shows that $\I$ and $\In_w(\I)$ have the same Hilbert series.

\subsection{Simplicial complexes and Stanley-Reisner ideals} \label{subsec:SRI}
A Stanley-Reisner ideal is a square-free monomial ideal.  Its combinatorial 
properties are governed by a simplicial complex.  
A simplicial complex $\Delta$ on vertices $\{1,\hdots,n\}$ is a collection of subsets of 
$\{1,\hdots, n\}$, called faces, that is closed under taking subsets. If $S \in \Delta$ has 
cardinality $k+1$, we call it a face of dimension $k$. A {\bf facet} of $\Delta$ is a face maximal in $\Delta$ under inclusion. 
Given a simplicial complex $\Delta$ on $[n]\backslash \{i\}$, define the cone of $\Delta$ over $i$ to be 
\[ 
{\rm cone}(\Delta, i) \ = \ \Delta \cup \{ S\cup \{i\} : S \in \Delta\}, \]
which is a simplicial complex on $[n]$ whose facets are in bijection with the facets of $\Delta$. 

\begin{definition}[See e.g. {\cite[Chapter II]{StanleyCCABook}}]Let $\Delta$ be a simplicial complex on vertices $\{1,\hdots, n\}$. The {\bf Stanley-Reisner ideal} of $\Delta$ is the square-free monomial ideal 
\[
\I_{\Delta} = \left\langle \x^S\ : \ S\subseteq [n], \ S\not\in \Delta \right\rangle
\]
generated by monomials corresponding to the non-faces of $\Delta$. The {\bf Stanley-Reisner ring} of $\Delta$ is the quotient ring $\C[\x]/\I_{\Delta}$. 
\end{definition}

The ideal $\I_{\Delta}$ is radical 
and it equals the intersection of prime ideals
\[
\I_{\Delta} = \bigcap_{F \text{ a facet of }\Delta} \langle x_i : i\not\in F \rangle.
\]
This writes the variety $\V(\I_{\Delta})$ as the union of coordinate subspaces  ${\rm span}\{e_i : i\in F\}$ 
where $F$ is a facet of $\Delta$.  In particular, if $\Delta$ has $k$ facets of dimension $d-1$, then 
$\V(\I_{\Delta})\subseteq \PP^{n-1}(\C)$ is a variety of dimension $d-1$ and degree $k$. See \cite[Chapter II]{StanleyCCABook}.

\subsection{Matroids}
Matroids are a combinatorial model for many types of independence relations. See \cite{Oxley} for 
general background on matroid theory. We can associate a matroid $M(\cL)$ to a linear space      $\cL\subset \C^n$ as follows.
Write a $d$-dimensional linear space $\cL\subset \C^n$ as the rowspan of a $d\times n$ matrix $A = (a_1, \hdots, a_n)$. 
A set $I \subseteq [n]$ is \textbf{independent} in $M(\cL)$ 
if the vectors $\{a_i : i\in I\}$ are linearly independent in $\C^d$. For any invertible matrix $U\in \C^{d\times d}$, the vectors $\{a_i : i\in I\}$
are linearly independent if and only if the vectors $\{Ua_i : i\in I\}$ are also independent, implying that this condition is independent of the choice of basis for $\cL$. 
Indeed, $I\subseteq [n]$ is independent in $M(\cL)$ if and only if the coordinate linear forms $\{x_i : i\in I\}$ are linearly independent when restricted to $\cL$.

Maximal independent sets are called \textbf{bases} and minimal dependent sets are called \textbf{circuits}. We use $\B(M)$ and $\CC(M)$ to denote the 
set of bases and circuits of a matroid $M$, respectively.  An element $i\in [n]$ is called a \textbf{loop} if $\{i\}$ is a circuit, and a \textbf{co-loop} if $i$ is contained in 
every basis of $M$. 
The \textbf{rank} of a subset $S \subseteq [n]$ is the largest size of an independent set in $S$. A \textbf{flat} is a set $F\subseteq [n]$ that is maximal for its rank, meaning that 
${\rm rank}(F) <{\rm rank}(F \cup \{i\}) $ for any $i\not\in F$.

Let $M$ be a matroid on $[n]$ and $i \in [n]$. The \textbf{deletion} of $M$ by $i$, denoted  $M\backslash i $, is the matroid
on the ground set $[n]\backslash i$ whose independent sets are subsets $I\subset [n]\backslash i$ that are independent in $M$. 
If $i$ is not a co-loop of $M$, then 
\[
\B(M\backslash i) = \{ B\in \mathcal{B}(M): i\notin B \} 
\ \ \text{ and }  \ \
\CC(M\backslash i) = \{C\in \mathcal{C}(M):i\notin C \}. 
\]
More generally, the deletion of $M$ by a subset $S\subset [n]$, denoted $M\backslash S$, is the matroid obtained from $M$ by 
successive deletion of the elements of $S$.  The {\bf restriction} of $M$ to a subset $S$, denoted $M|_S$, is 
the deletion of $M$ by $[n]\backslash S$.

If $i$ is not a loop of $M$, then the \textbf{contraction} of $M$ by $i$, denoted  $M/ i $, is the matroid
on the ground set $[n]\backslash \{i\}$ whose independent sets are subsets $I\subset [n]\backslash i$ 
for which $I\cup \{i\}$ is independent in $M$. 
Then 
\begin{align*}
\B(M/ i) &= \{B\backslash i \ :  \ B \in \mathcal{B}(M), i\in B \},  \text{ and }  \\
\CC(M/ i) & = \text{inclusion minimal elements of }\{C\backslash i \ : \ C\in \mathcal{C}(M) \}. 
\end{align*}
If $i$ is a loop of $M$, then we define the contraction of $M/i$ to be the deletion $M\backslash i$. 
The contraction of $M$ by a subset $S\subset [n]$, denoted $M/ S$, is obtained from $M$ by 
successive contractions by the elements of $S$.

For linear matroids, deletion and contraction correspond to projection and intersection in the following sense. 
For $S\subset [n]$, let $\cL\backslash S$ denote the linear subspace of $\C^{[n]\backslash S}$ obtained by projecting $\cL$ away 
from the coordinate space $\C^S = {\rm span}\{e_i : i\in S\}$. 
Let $\cL/S$ denote the intersection of $\cL$ with $\C^{[n]\backslash S}$.  Then 
\[ M(\cL)\backslash S = M(\cL\backslash S) 
\ \ \text{ and } \ \ 
M(\cL)/S = M(\cL/S). 
\]

Many interesting combinatorial properties of a matroid can be extracted from a simplicial complex 
called the broken circuit complex. Given a matroid $M$ and the usual ordering $1<2< \hdots <n$ on $[n]$, 
a \textbf{broken-circuit} of $M$ is a  subset of the form $C\backslash \min(C)$ where $C\in \CC(M)$. 
The \textbf{broken-circuit complex} of $M$ is the simplicial complex on $[n]$ whose faces are 
the subsets of $[n]$ not containing any broken circuit. 

\subsection{Reciprocal linear spaces}
For $I = [n]$, the variety $\inv_I(\cL)$ is well-studied in the literature. 
Proudfoot and Speyer study the coordinate ring of the variety $\inv_{[n]}(\cL)$ and relate it to the broken circuit complex of a matroid \cite{PS}. 
One of their motivations is connections with the cohomology of the complement of a hyperplane arrangement. 
These varieties also appear in the algebraic study of interior point methods for linear programming \cite{DLSV12} and entropy maximization for 
log-linear models in statistics \cite{MSUZ14}.  

If the linear space $\cL$ is invariant under complex conjugation, the variety $\inv_{[n]}(\cL)$ also has a special real-rootedness 
property.  Specifically, if $\cL^{\perp}$ denotes the orthogonal complement of $\cL$, then 
for any $u\in \R^n$, all the intersection points of $\inv_{[n]}(\cL)$ and the affine space $\cL^{\perp} + u$ are real. 
This was first shown in different language by Varchenko \cite{Var95} and used extensively in \cite{DLSV12}. 
One implication of this real-rootedness is that the discriminant of the projection away from $\cL^{\perp}$ is a nonnegative polynomial 
\cite{SSV13}.  Another is that $\inv_{[n]}(\cL)$ is a hyperbolic variety, in the sense of \cite{SV}. 
In fact, the Chow form of the variety $\inv_{[n]}(\cL)$ has a definite determinantal representation, certifying its hyperbolicity \cite{KV}. 
We generalize some of these results to $\inv_I(\cL)$.

In closely related work  \cite{AB16}, Ardila and Boocher study the closure of a linear space $\cL$ inside of $(\PP^1)^n$. 
For any $I\subseteq [n]$, $\inv_I(\cL)$ can be considered as an affine chart of this projective closure. 
Specifically, let $X \subseteq (\PP^1)^n$ denote the closure of the image of $\cL$ under $(x_1, \hdots, x_n)\mapsto ([x_1:y_1], [x_2:y_2], \hdots, [x_n:y_n])$, 
with $y_1 = \hdots = y_n = 1$.  The restriction of $X$ to the affine chart $x_i = 1$ for $i\in I$ and $y_j = 1$ for $j\in [n]\backslash I$ 
is isomorphic to $\inv_I(\cL)$. 
The cohomology and intersection cohomology of the projective variety $X$ have been studied to great effect in \cite{HuhWang} and \cite{ProudfootXuYoung}.
The precise relationship between the \emph{external activity complex} of a matroid used in  \cite{AB16} and the semi-broken circuit complex 
is described at the end of Section~\ref{sec:SimplicialComplex}.

\section{A semi-broken circuit complex}\label{sec:SimplicialComplex}

Let $M$ be a matroid on elements $[n]$ and suppose $I \subseteq [n]$. 
A vector $w \in \R^n$ with distinct coordinates gives an ordering on $[n]$, where $i<j$ whenever $w_i< w_j$. 
Without loss of generality, we can assume $w_1< \hdots <w_n$, which induces the usual order $1<\hdots <n$. 
Given a circuit $C$ of $M$ we define an \textbf{$I$-broken circuit} of $M$ to be  
\[
b_I(C) \ \ = \ \ 
\begin{cases}
C\backslash\min(C) & \text{if }C \subseteq I \\
(C\cap I) \cup \max(C\backslash I) & \text{if }C \not\subseteq I.
\end{cases}
\]
Now we define the \textbf{$I$-broken circuit complex} of $M$ to be 
\begin{equation}\label{eq:complexDef}
\Delta_w(M,I) \ =  \ \{ \tau \subseteq [n] \ : \ \tau \text{ does not contain an $I$-broken circuit of $M$}\}.
\end{equation}
Note that an $[n]$-broken circuit is a broken circuit in the usual sense and 
$\Delta_w(M,[n])$ is the well-studied broken circuit complex of $M$.

\begin{ex} \label{ex:NonGenMatroid} 
Consider the rank-$3$ matroid on $[5]$ with circuits $\mathcal{C} = \{124, 135, 2345\}$.
Let  $I = \{1,2,3\}$ and suppose $w\in (\R_+)^5$ with $w_1<\hdots < w_5$.
Then its $I$-broken circuits are 
$b_I(124) = 124$, $b_I(135) = 135$, and
$b_I(2345) = 235$.
The  simplicial complex $\Delta_w(M,I)$ is a pure $2$-dimensional simplicial complex with facets:
\[ {\rm facets}(\Delta_w(M,I))  \ \ = \ \ \{ 123, 125, 134, 145, 234, 245, 345   \}.\]
\end{ex}

The $I$-broken circuit complex shares many properties with the classical one, which will imply that 
it is always pure of dimension ${\rm rank}(M)-1$.

\begin{thm}\label{thm:Complex}
Let $\Delta_w(M,I)$ be the $I$-broken circuit complex defined in \eqref{eq:complexDef}. 
\begin{itemize}
\item[(a)] If $i\in I$ is a loop of $M$, then $\Delta_w(M,I) = \emptyset$. 
\item[(b)] If $i\in I$ is a coloop of $M$, then $\Delta_w(M,I)={\rm cone}(\Delta_w(M/i,I\backslash i),i)$.
\item[(c)] If $i=\max(I)$ is neither a loop nor a coloop of $M$, then 
\[\Delta_w(M,I)  \ = \ \Delta_w(M\backslash i,I \backslash i)  \ \cup \ {\rm cone}( \Delta_w(M/i,I\backslash i), i).   \]
\end{itemize}
\end{thm}
\begin{proof}
(a) If $i\in I$ is a loop, then $C = \{i\}$ is a circuit of $M$ with $b_I(C) = \emptyset$. 

(b) If $i\in I$ is a coloop, then no circuit of $M$, and hence no $I$-broken circuit, contains $i$. 
The circuits of $M$ are exactly the circuits of the contraction $M/i$ and the $I$-broken circuits of $M$ 
are the $(I\backslash i)$-broken circuits of $M/i$. 
Therefore $\tau$ is a face of $\Delta_w(M,I)$ if any only if $\tau\backslash i$ is a face of $\Delta_w(M/i,I\backslash i)$.

(c) ($\subseteq$) Let $\tau$ be a face of $\Delta_w(M,I)$.  We will show that if $i\not\in \tau$, then $\tau$ is a 
face of $\Delta_w(M\backslash i ,I\backslash i)$ and if $i \in \tau$, then $\tau\backslash i$ is a face of $\Delta_w(M/i ,I\backslash i)$.

If $i\not\in \tau$ and $C$ is a circuit of the deletion $M\backslash i$, then $C$
is a circuit of $M$, and $b_I(C)=b_{I\backslash i}(C)$ is an $I$-broken circuit of $M$ and therefore is not contained in $\tau$.  
If $i \in \tau$ and $C$ is a circuit of the contraction $M/i$, then either $C$ or $C\cup \{i\}$ is a circuit of $M$.  
In the first case, we again have that $b_I(C)=b_{I\backslash i}(C)$ is not contained in $\tau$ and thus not contained in $\tau\backslash i$.  
Secondly, suppose that $C\cup\{i\}$ is a circuit of $M$. 
If $C\subseteq I$, then $b_I(C\cup\{i\})$ is equal to $C\cup\{i\}\backslash \min(C\cup\{i\})$.  Since $i$ is the maximum element of $I$, 
this equals $C\backslash \min(C) \cup\{i\}$. This set is not contained in $\tau$. Therefore $b_{I\backslash i}(C) = C\backslash \min(C)$ 
is not contained in $\tau \backslash i$. 
If $C\not\subseteq I$, then the $I$-broken circuit of $C\cup \{i\}$ is $(C\cap I) \cup \{ i\} \cup \max(C\backslash I)$, 
which equals $b_{I\backslash i}(C) \cup \{i\}$.  Since $\tau$ cannot contain an $I$-broken circuit of $M$, $\tau\backslash i$ does not contain $b_{I\backslash i}(C)$. 

($\supseteq$) Let $\tau$ be a face of $\Delta_w(M\backslash i ,I\backslash i)$ and suppose $C$ is a circuit of $M$. 
If $i\not\in C$, then $C$ is also a circuit of $M\backslash i $, implying that $b_I(C)$ is not contained in $\tau$. 
If $i\in C$ and $C\subseteq I$, then $i = \max(C)$. Since $i$ is not a loop, this implies that $i\in b_I(C)$, which 
 cannot be contained in $\tau$. Similarly, if $i\in C$ and $C\not\subset I$, then $i\in b_I(C)$ and $b_I(C)\not\subset \tau$. 

Finally, let $\tau$ be a face of $\Delta_w(M/ i ,I\backslash i)$ and let $C$ be a circuit of $M$. 
If $i\in C$, then $C\backslash i $ is a circuit of $M/i$. Then 
$b_I(C)$ equals $b_{I\backslash i}(C\backslash i)\cup \{i\}$.  Since $\tau$ cannot contain $b_{I\backslash i}(C\backslash i)$, $\tau \cup\{i\}$ does not contain $b_I(C)$. 
If $i\not\in C$, then $C$ is a union of circuits of $M/i$, see \cite[\S 3.1, Exercise 2]{Oxley}. 
If $C\subseteq I$, then there is a circuit $C' \subseteq C$ of $M/i$ containing $\min(C)$. 
Then $C'\subseteq I\backslash i$ and $b_{I\backslash i}(C')$ is a subset of $b_I(C)$. 
Similarly, if $C\not\subseteq I$, then there is a circuit $C' \subseteq C$ of $M/i$ containing $\max(C\backslash I)$, 
giving $b_{I\backslash i}(C') \subseteq b_I(C)$. In either case, $\tau$ is a face of $\Delta_w(M/ i ,I\backslash i)$ and cannot contain
the broken circuit $b_{I\backslash i}(C')$ and therefore $\tau \cup \{i\}$ cannot contain $b_I(C)$. 
\end{proof}

\begin{cor} \label{cor:DeltaDim}
If $M$ is a matroid of rank $d$ with no loops in $I$, then $\Delta_w(M,I)$ is a pure simplicial complex of dimension $d-1$.
\end{cor}

\begin{proof}
We induct on the size of $I$. If $I = \emptyset$, then for every circuit $C$, the broken circuit $b_I(C)$ is 
the maximum element $\max(C)$. In this case, the simplicial complex $\Delta_w(M,I)$ consists of 
one maximal face $B$, where $B$ is the lexicographically smallest basis of $M(\cL)$.  
Here $B$ consists of the elements $i\in [n]$ for which the rank of $[i]$ in $M(\cL)$ is strictly larger than the rank of $[i-1]$.  
Every other element is the maximal element of some circuit of $M$. 

Now suppose $|I| >0$ and consider $i = \max(I)$.  If $i$ is a coloop of $M$, then the contraction $M/i$ is a matroid of rank $d-1$ with no loops in $I\backslash i$. 
Then by induction and Theorem~\ref{thm:Complex}(b), $\Delta_w(M,I) = {\rm cone}(\Delta_w(M/i,I\backslash i),i)$ is a pure simplicial complex of 
dimension $d-1$.  Finally, suppose $i$ is neither a loop nor a coloop of $M$. 
Then the deletion $M\backslash i$ is a matroid of rank $d$ and 
no element of $I\backslash i$ is a loop of $M\backslash i$. It follows that $\Delta_w(M\backslash i, I\backslash i)$ is a pure simplicial complex of 
dimension $d-1$.  The contraction $M/i$ is a matroid of rank $d-1$, implying that $\Delta_w(M/i,I\backslash i)$ is either empty (if $I\backslash i$ contains a loop of $M/i$), 
or a pure simplicial complex of dimension $d-2$.  In either case the decomposition in Theorem~\ref{thm:Complex} finishes the proof.
\end{proof}

We can also see this via connections with the \emph{external activity complex} defined by Ardila and Boocher \cite{AB16}.
Following their convention, for subsets $S,T\subseteq [n]$,  we use $x_Sy_T$ to denote the set $\{x_i:i\in S\}\cup \{y_j:j\in T\}$.

\begin{definition}\cite[Theorem 1.9]{AB16}
Let $M$ be a matroid and suppose $u\in \R^n$ has distinct coordinates. Then $u$ induces an order on $[n]$ where $i<j$ if and only if $u_i<u_j$. 
The \textbf{external activity complex} $B_u(M)$ is the
simplicial complex on the ground set $\{x_i,y_i:i\in [n]\}$ whose minimal non-faces 
are $\{x_{\min_{<_u}(C)}y_{C\backslash \min_{<_u}(C)}: C\in \mathcal{C} \}$.
\end{definition}

Given a weight vector $w\in (\R_{\geq 0})^n$ with distinct coordinates, define $u\in \R^n$ 
\[
u_i  = w_i \ \text{ for }i\in I \ \ \ \ \text{ and  }\ \ \ \  u_j  = -w_j \ \text{ for }j\not\in I.  \ 
\]
With this translation of weights, we can realize the semi-broken circuit complex $\Delta_w(M,I)$ as 
the link of a face in the external activity complex $B_u(M)$.  
Formally, the \textbf{link} of a face $\sigma$ in the simplicial complex
  $\Delta$ is the simplicial complex
  \[
   {\rm link}_{\Delta}(\sigma)=\{\tau \in \Delta : \tau \cup \sigma
    \in \Delta \text{ and } \tau \cap \sigma =\emptyset \}.
  \]
It is the set of faces that are disjoint from $\sigma$ but whose unions with
$\sigma$ lie in $\Delta$.

\begin{prop} Define weight vectors $u,w\in \R^n$ as above. If the matroid $M$ has no loops in $I$, then 
the semi-broken circuit complex $\Delta_w(M,I)$ is isomorphic to the link of 
the face $x_Iy_{[n]\backslash I}$ in the external activity
complex $B_u(M)$.
\end{prop}

\begin{proof}
First we show that $\sigma = x_Iy_{[n]\backslash I}$ is
actually a face of  $B_u(M)$ by arguing that $\sigma$ does not contain the minimal
non-face $x_{\min_{<_u}(C)}y_{C\backslash \min_{<_u}(C)}$ for any
circuit $C$ of $M$.
If $C$ is contained in $I$, then so is
$C\backslash \min_{<_u}(C)$. 
Since $M$ has no loops in $I$, this is nonempty and we can take $i\in C\backslash \min_{<_u}(C)$.
Then $y_i$ belongs to the non-face $x_{\min_{<_u}(C)}y_{C\backslash \min_{<_u}(C)}$, but not $\sigma$. 
On the other hand, if $C$ is not contained in $I$, we know $\min_{<_u}(C)$ is
contained in the complement of $I$, since the weight vector entries
satisfy $u_j<u_i$ for all $i\in I$ and $j\not\in I$.
Hence $x_{\min_{<_u}(C)}$, an element of the minimal
non-face associated to $C$,  does not belong to
$\sigma$.

Now we argue that $\Delta_w(M,I)$ and the link of $\sigma$ in $B_u(M)$ are isomorphic by 
identifying their non-faces. Note that the link of $\sigma$ is supported on the vertex set 
$x_{[n]/I}y_I$.  The bijection of vertices is then just $j\leftrightarrow x_j$ for $j\not\in I$ and $i \leftrightarrow y_i$ for $i\in I$. Note that $\tau\subseteq x_{[n]\backslash I}y_{I}$ is a face of the link 
of $\sigma$ in $B_u(M)$ if and only if for every circuit $C$, $\tau$ does not contain the intersection 
of the non-face $x_{\min_{<_u}(C)}y_{C\backslash \min_{<_u}(C)}$ with $x_{[n]\backslash I}y_{I}$. 
It suffices to check that these intersections are exactly the $I$-broken circuits of $M$. 

If $C$ is contained in $I$, then $w$ and $u$ give the same order on elements of $C$ and 
 $b_I(C)$  equals $ C\backslash\min_{<_u}(C)= C\backslash\min_{<_w}(C)$. Since $x_{\min_{<_u}(C)}\in \sigma$, we find that 
 \[  
 b_I(C) \ = \  \textstyle{C\backslash\min_{<_u}(C)} \ \ \leftrightarrow \ \ y_{C\backslash\min_{<_u}(C)} \ = \ x_{\min_{<_u}(C)}y_{C\backslash\min_{<_u}(C)} \backslash \sigma.
 \]
If $C$ is not contained in $I$, then  $b_I(C)$ equals $(C\cap I) \cup \max_{<_w}(C\backslash I)$. 
Since $u$ reverses the order on $[n]\backslash I$, this equals
$(C\cap I) \cup \min_{<_u}(C\backslash I)$. Then  
 \[  
 b_I(C) \ = \  \textstyle{(C\cap I) \cup \min_{<_u}(C\backslash I)} \ \ \leftrightarrow \ \ x_{\min_{<_u}(C\backslash I)}y_{C\cap I} \ = \ x_{\min_{<_u}(C)}y_{C\backslash\min_{<_u}(C)} \backslash \sigma,
 \]
where the equality $x_{\min_{<_u}(C\backslash I)} = x_{\min_{<_u}(C)}$  holds because 
$u_j < u_i$ for all $i\in I$ and $j\not\in I$.  
This shows that under this bijection of the vertices, the semi-broken circuit complex equals the link of $\sigma$ in the external activity complex.
\end{proof}

\begin{cor}
 The semi-broken circuit complex is shellable.
\end{cor}
\begin{proof}
In \cite{ACS16}, Ardila, Castillo, and Sampler show that the external activity complex, $B_u(M)$, is shellable. 
Then by \cite[Prop. 10.14]{BW97}, the link of any face in $B_u(M)$ is also shellable. 
\end{proof}

\begin{ex}
  Let $M$ be the matroid from Example~\ref{ex:NonGenMatroid}, $I = \{1,2,3\}$, and $u$ be
  the weight vector associated to $w$ as described above. 
  It induces the linear order $5<4<1<2<3$ on the
  ground set of the matroid $M(\cL)$.
  
  We outline the connection between the external activity complex $B_u(M)$
  and the semi-broken circuit complex by tracking two bases 
  $B_1=\{1,3,4\}, B_2=\{2,3,5\}$ of the matroid $M(\cL)$ in the
  construction of the two simplicial complexes.   
  For each basis, we split the complement $[5]\backslash B_i$ into externally
  active and externally passive elements. (See \cite[\S 2.5]{AB16} for the definitions of externally active and passive.) 
  For $B_1$, $\{2\}$ is externally
  passive and $\{5\}$ is externally active. Then by \cite[Theorem 5.1]{AB16}, the associated facet
  of $B_u(M)$ is $F_1=x_1x_2x_3x_4y_1y_3y_4y_5.$ By deleting $\sigma =
  x_1x_2x_3y_4y_5$ from $F_1$, we obtain the facet $x_4y_1y_3$ of
  ${\rm link}_{\Delta}(\sigma)$, corresponding to the facet $\{1,3,4\}$ of $\Delta_w(M,I)$. 
  For $B_2=\{2,3,5\}$, the externally
  passive elements are the entire complement $\{1,4\}$, hence the
  associated facet of $B_u(M)$ is $F_2=x_1x_2x_3x_4x_5y_2y_3y_5$. 
  Since $F_2$ does not contain $\sigma$, it does not contribute a facet to the link of $\sigma$ in $B_u(M)$.
 \end{ex}

The connection between this simplicial complex and the semi-inverted linear space $\inv_I(\cL)$ 
is that when $w \in (\R_+)^n$ has distinct coordinates, the ideal generated by the initial forms $\{ \In_w(f_C) : \text{ $C$ is a circuit of }M(\cL)\}$ 
is the Stanley-Reisner ideal of $\Delta_w(M,I)$. In fact, the initial form of $f_C$ is  $\In_w(f_C) =\x^{b_I(C)}$.
The ideal generated by these initial forms is then the Stanley-Reisner ideal $\mathcal{I}_{\Delta_w(M,I) }  = \langle \In_w(f_C):C \in \mathcal{C}(M) \rangle$.

\section{Proof of Theorem~\ref{thm:main}}\label{sec:proof}

In this section, we prove Theorem~\ref{thm:main}. To do this, we first use a flat degeneration of $\inv_I(\cL)$ 
to establish a recursion for its degree. 

\begin{prop}\label{prop:DegreeRecurrence}
Suppose $\cL$ is a linear subspace of $\C^n$ and $I\subseteq [n]$. 
Let $D(\cL,I)$ denote the degree of the affine variety $\inv_I(\cL)$. 
\begin{itemize} 
\item[(a)] If $i \in I$ is a loop of $M(\cL)$, then $\inv_I(\cL)$ is empty and $D(\cL,I) = 0$. 
\item[(b)] If $i\in I$ is a co-loop of $M(\cL)$, then $D(\cL,I) = D(\cL/i, I\backslash i)$. 
\item[(c)] If $i\in I$ is neither a loop nor a coloop of $M(\cL)$ then
\[
D(\cL\backslash i,I\backslash i) \ + \ D(\cL/i, I\backslash i) \ \ \leq \ \ D(\cL,I).
\]
\end{itemize}
\end{prop}
The proof of Theorem~\ref{thm:main} will show that there is actually equality in part (c). 

\begin{proof} Without loss of generality, take $i = 1$. 

(a) If $1\in I$ is a loop of $M(\cL)$ then $\cL$ is contained in the hyperplane $\{x_1=0\}$. 
Therefore the map $\inv_I$ is undefined at every point of $\cL$ and the image $\inv_I(\cL)$ is empty. 
By convention, we take the degree of the empty variety to be zero. 

(b) If $1$ is a co-loop of $M(\cL)$, then $\cL$ is a direct sum of ${\rm span}(e_1)$ and $\cL/1$, meaning that 
any element in $\cL$ can be written as $ae_1 + b$ where $a\in \C$ and $b\in \cL/1$. 
For points at which the map $\inv_I$ is defined, $\inv_I(ae_1 + b) = a^{-1}e_1 + \inv_{I\backslash 1}(b)$. 
From this, we see that $\inv_I(\cL)$ is the direct sum of ${\rm span}(e_1)$ and $\inv_{I\backslash 1}(\cL/1)$, 
implying that $\inv_{I}(\cL)$ and $\inv_{I\backslash 1}(\cL/1)$ have the same degree. 

(c) 
Let $\I$ denote the ideal of polynomials vanishing on $\inv_I(\cL)$ and 
$\J = \ol{\I}$ denote its homogenization in $\C[x_0,x_1,\hdots, x_n]$. 
Take $w = e_1 \in \R^{n+1}$ and consider $\In_w(\J)$, 
as defined in Section~\ref{subsec:Grobner}. 
We will show that the variety of $\In_w(\J)$ contains 
the image in $\PP^n$ of both $\{0\}\times \inv_{I\backslash 1}(\cL\backslash 1)$ and 
$\mathbb{A}^1(\C)\times \inv_{I\backslash 1}(\cL/ 1)$. 
Since both these varieties have dimension equal to $\dim(\cL)$, the degree of the variety of $\In_w(\J)$ is at least the sum of their degrees. 
The claim then follows by the equality of the Hilbert series of $\J$ and $\In_w(\J)$.

If $j\in I\backslash 1$ is a loop of $M(\cL)$, then $j$ is a loop of $M(\cL\backslash 1)$ and $D(\cL\backslash 1,I\backslash 1)=0$.
Otherwise the set $U_I$ is Zariski-dense in $\cL$, where $U_I$ denotes the intersection of 
$\cL$ with $(\C^*)^I\times \C^{[n]\backslash I}$. 

Let $\pi_I$ denote the coordinate projection $\C^n \rightarrow \C^I$. On $U_I$, the maps 
$\pi_{I\backslash 1}\circ \inv_I$ and $\inv_{I\backslash 1} \circ \pi_{I\backslash 1}$ are equal:
\[\pi_{I\backslash 1}(\inv_I(x))  \ \ =  \ \ \inv_{I\backslash 1}(\pi_{I\backslash 1}(x)) \ \ = \ \
\sum_{j\in I\backslash 1} x_j^{-1} e_j + \sum_{j\not\in I} x_j e_j. 
\]
In particular, the points $\inv_{I\backslash 1}(\pi_{I\backslash 1}(U_I))$ are Zariski dense in $\inv_{I\backslash 1}(\cL\backslash 1)$. 
Now let  $x$ be a point of $ \inv_I(U_I)$. 
Then $[1: x]$ belongs to the variety of $\J$ and, for every $t\in \C$, the point $(t,t^{e_1}\cdot[1:x])$ belongs to the 
variety of $\ol{\J}^{w}$, as defined in Section~\ref{subsec:Grobner}. Taking $t\rightarrow 0$, we see that $[1:0:\pi_{I\backslash 1}(x)]$ belongs to the variety of $\In_{e_1}(\J)$.

If $j\in I\backslash 1$ is a loop of $M(\cL/1)$, then $\inv_{I\backslash 1}(\cL/1)$ is empty and 
the claim follows. Otherwise the intersection $U_{I\backslash 1}$ of $\cL/1$ with $\{0\}\times (\C^*)^{I\backslash 1} \times \C^{[n]\backslash I}$ is nonempty and Zariski-dense in $\cL\cap \{x\in \C^n: x_1=0\} \cong \cL/1$. 
Let $x \in U_{I\backslash 1}$. Since $1$ is not a loop of $M(\cL)$, 
there is a point $v \in \cL$ with $v_1 = 1$. Then for any $\lambda, t\in \C^*$,
$x+ (t/\lambda)v$ belongs to $\cL$ and for all but finitely many values of $t$, $y(t) = \inv_I(x+ (t/\lambda)v)$
is defined and has first coordinate $y_1(t) = \lambda/t$. 
Then $[1: y(t)] \in \V(\J)$ and $(t, t^{e_1}\cdot [1:y(t)])$ belongs to $\V\left(\ol{\J}^w\right)$. 
Note that the limit of $t^{e_1}\cdot [1:y(t)] = [1:\lambda : y_2(t): \hdots :y_n(t)]$ as $t\rightarrow 0$ equals $[1:\lambda: \inv_{I\backslash 1}(x)]$. 
Therefore for every point $(\lambda,u)\in \mathbb{A}^1(\C)\times \inv_{I\backslash 1}(\cL/1)$, the point $[1:\lambda: u]$ belongs to $\V\left(\In_{e_1}(\J)\right)$. \end{proof}

We also need the following fact from commutative algebra, included here for completeness, 
which may be clear to readers familiar with the geometry of schemes. 
Recall that for a homogeneous ideal $J \subset \C[\x]$, the Hilbert polynomial of $J$ is 
the polynomial $h(t)$ that agrees with $\dim_\C(\C[\x]/J)_t$ for sufficiently large $t \in \N$. 
Then $ h(t) =  \sum_{i=0}^d b_i \binom{t}{d-i}$ for some $d\in \Z_{>0}$, where $b_i\in \Z$ with $b_0>0$.
In a slight abuse of notation, we say that the dimension of $J$ is $\dim(J) = d$ and the degree of $J$ is $\deg(J) = b_0$. 
The ideal $J$ is \emph{equidimensional} of dimension $d$ if $\dim(P)=d$ for every minimal associated prime $P$ of $J$.

\begin{lem} \label{lem:EquidimEquality}
Let $I\subseteq J\subseteq\C[\x]$ be equidimensional homogeneous ideals of dimension $d$. If $I$ is radical and $\deg(I)\leq \deg(J)$, then $I$ and $J$ are equal. 
\end{lem}
\begin{proof}
Let $I = P_1 \cap \hdots \cap P_r$ and $J = Q_1 \cap \hdots \cap Q_s$ be irredundant primary decompositions of $I$ and $J$.
Without loss of generality, we can assume that $\dim(Q_i) = d$ for $1 \leq i \leq u$, and since $\V(J) \subseteq \V(I)$, the prime ideals $P_i$ can be reindexed such that $P_i=\sqrt{Q_i}$, 
implying $Q_i \subseteq P_i$. 
For all $1 \leq i \leq u$, there exists an element $a\in (\cap_{j\neq i}P_j) \cap (\cap_{j\neq i}\sqrt{Q_j})$ with $a\not\in P_i$. 
Then the saturation $I:\langle a\rangle^{\infty} = P_i $ is contained in $ J:\langle a\rangle^{\infty} = Q_i$, implying $P_i = Q_i$. 
This writes the ideal $J$ as $J=P_1\cap \ldots \cap P_{u}\cap Q_{u+1} \cap \hdots \cap Q_s$. 
The degree of an ideal is equal to the sum of the degrees of the top dimensional ideals in its primary decomposition, hence
\[
\deg(I) = \sum_{i=1}^r \deg(P_i) \ \ \ \text{ and } \ \ \ \
\deg(J) = \sum_{i=1}^u \deg(Q_i)=\sum_{i=1}^u \deg(P_i).
\]
The assumption that $\deg(I)\leq \deg(J)$ implies that $r=u$, which gives the reverse containment
$I = P_1 \cap \hdots \cap P_u \supseteq J$.
\end{proof}

\begin{proof}[Proof of Theorem~\ref{thm:main}] 
We proceed by induction on $|I|$. If $|I|=0$, then $\inv_I(\cL)$ is just the linear space $\cL$. 
Then Theorem~\ref{thm:main} reduces to the statement that the linear forms supported on circuits 
form a universal Gr\"obner basis for $\I(\cL)$. See e.g. \cite[Prop. 1.6]{GrobnerPolytope}.

Now take $|I|\geq 1$,  $w\in (\R_+)^n$ with distinct coordinates, and let $M$ denote the matroid $M(\cL)$.
If $M$ has a loop $i$ in $I$, then for the circuit $C = \{i\}$, the circuit polynomial $f_C$ equals $1$, 
which is a Gr\"obner basis for the ideal of polynomials vanishing on the empty set $\inv_I(\cL)$. 
Therefore we may suppose that $M$ has no loops in $I$, in which case $\inv_I(\cL)$ is a $d$-dimensional 
affine variety of degree $D(\cL,I)$. 

Let $\Delta$ denote the $I$-broken circuit complex $\Delta_w(M,I)$
defined in Section~\ref{sec:SimplicialComplex} and let $\Delta_0$ denote the simplicial complex on elements $\{0,\hdots, n\}$ 
obtained from $\Delta$ by coning over the vertex $0$. Let $\I_{\Delta_0}$ denote the Stanley-Reisner ideal 
of $\Delta_0$, as in Section~\ref{subsec:SRI}.

Let $\I \subset \C[\x]$ be the ideal of polynomials vanishing on $\inv_I(\cL)$ 
and define the ideal $\J\subset \C[x_0, x_1, \hdots, x_n]$ to be its homogenization with respect to $x_0$. 
Since $\inv_I(\cL)$ is the image of an irreducible variety under a rational map, it is also irreducible. 
It follows that the ideals $\I$ and $\J$ are prime. 
For a circuit polynomial $f_C$, its homogenization $\ol{f_C}$ belongs to $\J$ and since $w\in (\R_+)^n$, 
\[
\In_{(0,w)}(\ol{f_C})  =  \In_{w}(f_C) = \begin{cases}
 a_k \x^{C\backslash k}   & \text{ if $C \subseteq I $ and $k = {\rm argmin}\{w_j: j\in C\}$ } \\ 
 a_k  \x^{C\cap I\cup k}  & \text{ if $C \not\subseteq I $ and $k = {\rm argmax}\{w_j: j\in C\backslash I\}$. } \\ 
\end{cases}
\]
Up to a scalar multiple, $\In_{w}(f_C)$ equals the 
square-free monomial corresponding to the $I$-broken circuit of $C$, namely $\x^{b_I(C)}$. 
It follows that 
\[
\langle \In_{w}(f_C) \ : \ C \in \mathcal{C}(M) \rangle  \ = \ \I_{\Delta}
\ \ 
\text{ and } 
\ \ 
\langle \In_{(0,w)}(\ol{f_C}) \ : \ C \in \mathcal{C}(M) \rangle  \ = \ \I_{\Delta_0}. 
\]
From this we see that $\I_{\Delta_0} \subseteq \In_{(0,w)}(\J)$.

Let $i = {\rm argmax}\{w_j : j\in I\}$.  By the inductive hypothesis, $D(\cL\backslash i, I\backslash i)$ and $D(\cL/ i, I\backslash i)$
are the number of facets of $\Delta_w(M\backslash i, I\backslash i)$ and $\Delta_w(M/i, I\backslash i)$, respectively. 
Therefore by Theorem~\ref{thm:Complex}, $\Delta$ and thus $\Delta_0$ each have 
$D(\cL/ i, I\backslash i)$ facets if $i$ is a coloop of $M$ and $D(\cL\backslash i, I\backslash i) + D(\cL/ i, I\backslash i)$ facets otherwise. 
Then by Proposition~\ref{prop:DegreeRecurrence}, $\Delta_0$ has at most $D(\cL, I)$ facets and the Stanley-Reisner ideal 
$\I_{\Delta_0}$ has degree $\leq D(\cL, I)$.

Since $\Delta_0$ is a pure simplicial complex of dimension $d$, $\I_{\Delta_0}$ is an 
equidimensional ideal of  dimension $d$. 
As $\J$ is a prime $d$-dimensional ideal, its initial ideal $\In_{(0,w)}(\J)$  is equidimensional of the same 
dimension, see \cite[Lemma 2.4.12]{tropBook}.  

The ideals $\I_{\Delta_0}$ and $\In_{(0,w)}(\J)$ then satisfy the hypotheses of Lemma~\ref{lem:EquidimEquality}, 
and we conclude that they are equal.  By \cite[Prop. 2.6.1]{tropBook}, restricting to $x_0=1$ 
gives that 
\[
\I_{\Delta} \ \ = \ \ \langle \In_w(f_C) \ : \ C\in \mathcal{C}(M) \rangle \ \  = \ \ \In_w(\I).  
\]
As this holds for every $w\in (\R_+)^n$ with distinct coordinates, it will also hold for arbitrary $w \in (\R_{\geq 0})^n$ (see \cite[Prop. 1.13]{GrobnerPolytope}).
It follows from \cite[Cor. 1.9, 1.10]{GrobnerPolytope} that the circuit polynomials $\{f_C : C\in \mathcal{C}(M)\}$ form a universal Gr\"obner basis for $\I$. 
\end{proof}

\begin{ex} \label{ex:NonGen} Consider the $3$-dimensional linear space in $\C^5$: 
\[ \cL  \ \ = \ \ {\rm rowspan}\begin{pmatrix} 
1&0&0&1&1\\
0&1&0&1&0\\
0&0&1&0&1
\end{pmatrix}.\]
The circuits of the matroid $M(\cL)$ are $\mathcal{C} = \{124, 135, 2345\}$. Take $I = \{1,2,3\}$. Then 
\[ f_{124} = x_1 + x_2 - x_1x_2x_4, f_{135} =x_1+x_3 -x_1x_3x_5,  \text{ and }  f_{2345} = x_2 - x_3+x_2x_3x_4 - x_2x_3x_5. \]
If $w\in (\R_+)^5$ with $w_1<\hdots < w_5$, then the ideal $\langle \In_w(f_C): C\in \mathcal{C}\rangle$ is 
$\langle x_1x_2x_4, x_1x_3x_5, x_2x_3x_5 \rangle$.
The  simplicial complex $\Delta_w(M,I)$ is $2$-dimensional and has seven facets:
\[ {\rm facets}(\Delta_w(M,I))  \ \ = \ \ \{ 123, 125, 134, 145, 234, 245, 345   \}.\]
Indeed, the variety of $\langle x_1x_2x_4, x_1x_3x_5, x_2x_3x_5 \rangle$ is the union the seven coordinate linear spaces 
${\rm span}\{e_i, e_j, e_k\}$ where $\{i,j,k\}$ is a facet of $\Delta_w(M,I)$. 

Interestingly, it is not true that the homogenizations $\ol{f_C}$ form a universal Gr\"obner basis for the homogenization $\ol{\I}$. 
Indeed, consider the weight vector $(2,0,0,1,1,1)$. The ideal generated by the initial forms of circuit polynomials  
$\langle \In_{w}(\ol{f_C}):C\in \mathcal{C}\rangle$ is $\langle
x_0^2x_1+x_0^2x_2,x_0^2x_3\rangle$, whereas 
$\In_{w}(\ol{\I})=\langle x_2x_3x_4-x_1x_3x_5-x_2x_3x_5,x_0^2x_1+x_0^2x_2,x_0^2x_3 \rangle$. 
Nevertheless, upon restriction to $x_0=1$, the two ideals become equal. 
\end{ex}

\begin{cor}\label{cor:degreeFacets}
If $\dim(\cL) = d$, then the \emph{affine} Hilbert series of the ideal $\I \subseteq \C[\x]$ of polynomials vanishing on $\inv_I(\cL)$
 is
\[
\sum_{m=0}^{\infty} \dim_{\C}(\C[\x]_{\leq m}/\I_{\leq m}) \ t^m   
=   \frac{1}{(1-t)^{d+1}} \sum_{i=0}^{d}f_{i-1} t^i(1-t)^{d-i}
=   \frac{ h_0 + h_1 t + \hdots + h_dt^d}{(1-t)^{d+1}}.
\]
where $(f_{-1}, \hdots, f_{d-1})$ and $(h_0,\hdots, h_d)$ are the $f$- and  $h$-vectors of $\Delta_w(M,I)$. 
In particular, its degree is the number of facets $f_{d-1} = h_0 + h_1 + \hdots + h_d$. 
\end{cor}      
\begin{proof}
The affine Hilbert series of $\I$ equals the classical Hilbert series of its homogenization $\ol{\I}$, 
which equals the Hilbert series of $\In_{(0,w)}(\ol{\I})$ for any $w\in \R^{n}$. 
When the coordinates of $w$ are distinct and positive, $\In_{(0,w)}(\ol{\I})$ is the Stanley-Reisner ideal of 
$\Delta_0 = {\rm cone}(\Delta_w(M,I),0)$.  Since the Stanley-Reisner ideals of $\Delta = \Delta_w(M,I)$ and $\Delta_0$ 
are generated by the same square-free monomials, their Hilbert series differ by a factor of $1/(1-t)$. 
The result then follows from well known formulas for the Hilbert series of $\I_\Delta$, \cite[Ch. 1]{CCABook}. 
\end{proof}

The proof of Theorem~\ref{thm:main} shows that there is equality in Proposition~\ref{prop:DegreeRecurrence}(c), namely that 
if $i\in I$ is neither a loop nor a coloop of $M(\cL)$, then the degree $D(\cL, I) $ satisfies $D(\cL, I) = D(\cL\backslash i, I\backslash i ) + D(\cL/ i, I\backslash i )$. 
From this we can derive an explicit formula for the degree of $\inv_I(\cL)$ in the uniform matroid case.

\begin{cor}\label{cor:genericDegree}
For a generic $d$-dimensional linear space $\cL \subseteq \C^n$ and $I\subseteq [n]$ of size $|I|=k$, the degree of $\inv_I(\cL)$ equals  
\[
D(\cL,I)= \sum_{j=k+d - n}^{d}\binom{k}{j} -\binom{k-1}{d} ,
\]  
where we take $\binom{a}{b}=0$ whenever $a < 0$ or $b<0$.  In particular, for $n \geq k+d$, the degree only depends on $d$ and $k$. 
\end{cor}
\begin{proof}
By assumption $k,d,n$ satisfy the inequalities $0\leq  k \leq n$ and $0\leq  d \leq n$.
We proceed by induction on $k$. In the extremal cases, $D(\cL, I)$ satisfies 
\[
D(\cL,I) \ = \ 
\begin{cases}
1 & \text{if }k=0, \\
1 & \text{if } d=n, \\
0 & \text{if }d= 0 \text{ and } k \geq 1.  \\
\end{cases}
\]
Indeed, if $I = \emptyset$, then $\inv_I(\cL) = \cL$ and $D(\cL,I) = 1$. If $d=n$, then $\inv_I(\cL)$ is all of $\C^n$ and $D(\cL,I) = 1$. 
Finally, if $d=0$ and $|I|\geq 1$, then $n\geq |I| \geq 1$, and $\cL= \{(0,\hdots, 0)\}$ in $\C^n$. The map $\inv_I$ is not defined at this point so $\inv_I(\cL)$
is empty and thus has degree $0$. 

Suppose $k\geq 1$ and $0 <d < n$. 
Since any $i\in I$ is neither a loop nor a coloop, $D(\cL, I) $ equals $D(\cL\backslash i, I\backslash i ) + D(\cL/ i, I\backslash i )$ by Proposition~\ref{prop:DegreeRecurrence}(c) and the proof of Theorem~\ref{thm:main}. Recall that $\cL\backslash i$ and $\cL/i$ are subspaces in $\C^{n-1}$ of dimensions $d$ and $d-1$, respectively. 
Since $|I\backslash i| = k-1$, by induction we get that 
\begin{align*}
 D(\cL\backslash i, I\backslash i ) &= \sum_{j=k+d - n}^{d}\binom{k-1}{j} -\binom{k-2}{d},
 \text{ and } \\
  D(\cL/i, I\backslash i ) &=  \sum_{j=k+d - n-1}^{d-1}\binom{k-1}{j} -\binom{k-2}{d-1}.
\end{align*}
Since $\binom{k-1}{j} + \binom{k-1}{j-1}  = \binom{k}{j}$ and $\binom{k-2}{d} + \binom{k-2}{d-1}= \binom{k-1}{d}$, 
the sum $D(\cL\backslash i, I\backslash i ) + D(\cL/ i, I\backslash i )$ is the desired formula for $D(\cL, I)$. 
\end{proof}

\begin{ex} The number of facets of the complex $\Delta_w(M,I)$ gives the degree $D(\cL,I)$ and if $M$ is the uniform matroid 
of rank $d$ on $[n]$, we can write out these facets explicitly.  Let $w = (1,\hdots, n)$ and consider $I = \{1,\hdots, k\}$. 
If $k\leq d$, no circuit is contained in the inverted set $I$, implying that every broken circuit has the form 
$(C\cap I) \cup \max\{C\backslash I\}$. Then every facet of $\Delta_w(M,I)$ has the form $S \cup \{k+1, \hdots,  k+d-j\}$ 
where $S\subseteq I$ and $|S| = j \leq d$. For fixed $j$, the number of possibilities are $\binom{k}{j}$, and the
constraints on $j$ are $k+ d - j \leq n$ and $0 \leq j \leq k \leq d$.  
If $k > d$, then every subset of $\{2,\hdots, k\}$ of size $d$ is an $I$-broken circuit. 
From the list of facets $S \cup \{k+1, \hdots,  k+d-j\}$, we remove those for which $S\subset \{2,\hdots, k\}$ 
and $|S| = d$, of which there are $\binom{k-1}{d}$.  
\end{ex}

\section{Supports }\label{sec:Supports}
In this section, we characterize the intersection of the variety $\inv_I(\cL)$ with the coordinate hyperplanes. 
These are exactly the points in the closure of, but not the actual image of, the map $\inv_I$. 
Given a point ${\bf p}\in \C^n$, its support is the set of indices of its nonzero coordinates: $\supp({\bf p}) = \{i : p_i \neq 0\}$. 
For a subset $S\subseteq [n]$, we will use $\C^S$ to denote the set of points ${\bf p}$ with $\supp({\bf p}) \subseteq S$
and $\ol{S}$ to denote the complement $[n]\backslash S$.

\begin{thm} \label{thm:supp}
Suppose that the matroid $M=M(\cL)$ has no loops in $I$.
For $S\subseteq [n]$, let $T = S\cup \ol{I}$. 
If $T$ is a flat of $M$, then the restriction of $\inv_I(\cL)$ to $\C^S$ is given by 
\[
\inv_I(\cL)\cap \C^S \ \ = \ \ \inv_{S\cap I}\left(\pi_{T}(\cL)\cap \C^S\right),
\] 
where $\pi_T$ denotes the coordinate projection $\C^n \rightarrow \C^T$. 
Moreover,  $\supp(\p) = S$ for some $\p\in \inv_I(\cL)$ if and only if 
$T$ is a flat of $M$ and $T\backslash S$ is a flat of $M|_T$. 
\end{thm}

We build up to the proof of Theorem~\ref{thm:supp} by considering the cases $\ol{I} \subseteq S$ and 
$I \subseteq S$. 

\begin{lem} \label{lem:pi}
If $S\subseteq [n]$ is a flat of $M$ with $\ol{I} \subseteq S$, then 
\[
\inv_I(\cL)\cap \C^S \ \ = \ \ \inv_{S\cap I}\left(\pi_{S}(\cL)\right),
\] 
where $\pi_S$ denotes the coordinate projection $\C^n \rightarrow \C^S$. 
\end{lem}
\begin{proof}
Recall that $F\subset [n]$ is a flat of $M$ if and only if $|\ol{F}\cap C|\neq 1$ for all 
circuits $C$ of $M$. Suppose that $S$ is a flat of $M$ and consider 
 the restriction of the circuit polynomials $f_C$ to $\C^S$. 
Note that $\ol{S} \subseteq I$, so that for any circuit $C$ with $|C\cap \ol{S}| \geq 2$, 
 $|C\cap I|\geq 2$ and the circuit polynomial $f_C$ is zero at every point of $\C^S$. 
 
 The circuits for which $|C\cap \ol{S}|= 0$ are exactly the circuits contained in $S$, 
 which are the circuits of the matroid restriction $M|_S$. Moreover the projection 
 $\pi_S(\cL)$ is cut out by the vanishing of the linear forms 
 $\{\ell_C: C\in \CC(M), C\subseteq S\}$, which are exactly the linear forms  
 $\{\ell_{C'} : C'\in \CC(M|_S)\}$. 
 It follows that the circuit polynomials ${\{f_{C'} : C'\in \CC(M|_S)\}}$
 are a subset of the circuit polynomials of $\cL$, namely  
 ${\{f_C: C\in \CC(M), C\subseteq S\}}$.  By Theorem~\ref{thm:main}, the variety of 
 circuit polynomials is the variety of the semi-inverted linear space, giving that 
 \begin{align*}
 \inv_I(\cL)\cap\C^S  & = \V(\{f_C: C\in \CC(M), C\subseteq S\})\cap{\C^S} \\
 & = \V(\{f_{C'} : C'\in \CC(M|_S)\}) =  \inv_{S\cap I}(\pi_S(\cL)). \end{align*}
\end{proof}

\begin{lem} \label{lem:cap}
If $S\subseteq [n]$ with $I \subseteq S $, then 
$\inv_I(\cL) \cap \C^S =\inv_{I}\left(\cL \cap \C^S\right)$.  
\end{lem}
\begin{proof} 
($\supseteq$) The affine variety $\inv_{I}\left(\cL \cap \C^S\right)$ is the Zariski-closure of 
$\cL \cap \C^S$ under the map $\inv_I$.  Since $\cL \cap \C^S$ is contained in $\cL$, 
$\inv_{I}\left(\cL \cap \C^S\right)$ is a subset of $\inv_I(\cL)$. Moreover $\inv_I(\p) \in \C^S$ 
for any point $\p\in \C^S$. The inclusion follows. 

($\subseteq$) For this we show the reverse inclusion of the ideals of polynomials vanishing on these varieties. 
Now let $C'$ be a circuit of $M(\cL\cap \C^S)$ and 
$\ell_{C'} = \sum_{i\in C'}a_ix_i$ its corresponding linear form. 
Then for some circuit $C$ of $M$, 
$C' = C\cap S$ and $\ell_{C'} $ equals the restriction $\ell_C(\pi_S(\x))$.  
Applying $\inv_I$ and clearing denominators then gives 
\[
f_{C'}(\x) 
\ \ = \ \ \x^{C'\cap I} \ell_{C'}(\inv_I(\x))
\ \ = \ \  \x^{C\cap I} \ell_{C}(\inv_I(\pi_S(\x))) \ \ = \ \ f_C(\pi_S(\x)).
\]
The middle equation holds because $I \subseteq S$, which implies that $C\backslash C' \subseteq \ol{S} \subseteq \ol{I}$.
\end{proof}

\begin{proof}[Proof of Theorem~\ref{thm:supp}]
Suppose that $T$ is a flat of $M$. Since $\ol{I}\subseteq T$, Lemma~\ref{lem:pi} says that
the restriction $\inv_I(\cL)\vert_{\C^T}$ equals $\inv_{T \cap I}\left(\pi_{T}(\cL)\right)$.
Furthermore since $T\cap I = S \cap I \subseteq S$,  we can apply Lemma~\ref{lem:cap} 
to find the intersection of $\inv_{T \cap I}\left(\pi_{T}(\cL)\right)$ with $\C^S$. 
All together this gives that $\inv_I(\cL)\cap \C^S$  equals 
\begin{equation}\label{eq:projcap}
(\inv_I(\cL)\cap\C^T)\cap\C^S  =  \inv_{S\cap I}\left(\pi_{T}(\cL)\right)\cap\C^S = \inv_{S\cap I}\left(\pi_{T}(\cL)\cap \C^S\right). 
\end{equation}

Suppose further that $T\backslash S$ is a flat of the matroid $M|_T$.
This implies that the contraction of $M|_T$ by $T\backslash S$ has no loops. This is the matroid 
of the linear space $\pi_T(\cL)\cap \C^S$, which is therefore not contained in any coordinate subspace $\{x_i=0\}$ for $i\in S$. 
It follows that there is a point $\p\in \pi_T(\cL)\cap \C^S$ of full support $\supp(\p) = S$. 
Equation \eqref{eq:projcap} then shows that $\inv_{S\cap I}(\p)$ is a point of support $S$ in $\inv_I(\cL)$.

Conversely, 
suppose that $S = \supp(\p)$ for some point $\p \in \inv_I(\cL)$.  Then $T$ is a flat of $M$. 
To see this, suppose for the sake of contradiction that for some circuit $C$ of $M$, $C\cap\ol{T}=\{j\}$. 
 Then $j$ is the unique element of $C\cap I$ for which $p_j=0$, and 
 evaluating  the circuit polynomial $f_C$ at the point $\p$ gives 
\[
f_C(\p) \ = \ \sum_{i\in C\cap I} a_i \p^{C\cap I\backslash\{ i\}}  +  \sum_{i\in C\backslash I } a_i \p^{C\cap I \cup\{ i\}}
\ = \ a_j \p^{C\cap I\backslash\{ j\}} \ \neq  \ 0,
\]
contradicting $\p\in \inv_I(\cL)$. Therefore $T$ is a flat of $M$ and \eqref{eq:projcap} holds. 
It follows that $\p$, or more precisely $\pi_T(\p)$, is a point of support $S$ in $\pi_T(\cL)$. 
Therefore $\pi_T(\cL)\cap\C^S$ contains a point of full support, the contraction of the matroid $M|_T$ by $T\backslash S$ 
has no loops, and $T\backslash S$ is a flat of the matroid $M|_T$. 
\end{proof}

\begin{ex} Suppose $\cL$ is a generic $d$-dimensional subspace of $\C^n$, and hence that $M =M(\cL)$ is the uniform matroid 
of rank $d$ on $[n]$. Its flats are the subsets $F\subseteq [n]$ of size $|F|<d$, along with the full set $[n]$. 
Consider $S\subseteq [n]$ and $T = S \cup \ol{I}$.  If $T$ is a flat of $M$, then either $|T|<d$, implying $|\ol{I}|<d$, 
or $T= [n]$, in which case $I\subseteq S$. 
If $|T|<d$, then $M|_T$ is the uniform matroid of rank $|T|$ on the elements $T$. 
Then every subset of $T$ is a flat of $M|_T$ and $S$ is the support of a point in $\inv_I(\cL)$. 
If $T = [n]$, then $T\backslash S = \ol{S}$ is a flat of $M|_T = M$ if and only if 
$|\ol{S}|<d$ or $|\ol{S}|=n$. Since $S$ contains $I$, $|\ol{S}|=n$ only when $I = S=\emptyset$. 
Therefore if $I \neq \emptyset$, we have $|S|>n-d$. Putting these together gives
\[
S\in \supp(\inv_I(\cL))  \  \ \Longleftrightarrow \ \
\begin{cases} 
S =\emptyset \text{ or } |S|>n-d & \text{ if } I = \emptyset\\
I\subseteq S \text{ and } |S|>n-d & \text{ if } 0<|I|\leq n-d\\
 |S\cup \ol{I}| <d \text{ or } I\subseteq S & \text{ if }  n-d<|I|.\\
\end{cases}
\]
\end{ex}

\section{Real points and hyperplane arrangements}\label{sec:real}
Here we explore a slight variation of $\inv_I$ that preserves a real-rootedness property of certain intersections. 
Given a  polynomial $f\in \R[x_1, \hdots, x_n]$ with a real-rootedness property called \emph{stability}, 
it is known that the polynomial $x_1^{\deg_{e_{1}}(f)}\cdot f(-1/x_1, x_2, \hdots, x_n)$ is again stable  \cite[Lemma 2.4]{wagnerSurvey}. 
Here we extend that to an action preserving real-rootedness of intersections with a family of affine-spaces. 
For $I\subseteq [n]$, define the rational map $\inv_I^-: \C^n \dashrightarrow  \C^n $ by 
\[(\inv_I^-(x))_i =
\begin{cases}
-1/x_i  &\text{ if } i\in I\\
 \ \ x_i  &\text{ if } i\not\in I.\\
\end{cases}
\]
Equivalently this is the composition of $\inv_I$ with the map that scales coordinates $x_i$ for $i\in I$ by $-1$. 
Note that the varieties $\inv_I(\cL)$ and $\inv_I^-(\cL)$ are isomorphic, and in particular they have the same degree. 
For any linear space $\cL \subset \C^n$, let $\cL^{\perp}$ denote the subspace of vectors $v$ 
for which $\sum_{i=1}^nv_ix_i=0$ for all $x\in \cL$.

\begin{prop}\label{prop:realPoints}
If $\cL \subset \C^n$ is invariant under complex conjugation, 
then for any $u\in \R^n$, all of the intersection points of $\inv_I^-(\cL)$ with $\cL^{\perp} + u$
are real. 
\end{prop}

\begin{proof}
If $\cL$ is contained in a coordinate hyperplane $\{x_i=0\}$ where $i\in I$, then 
$\inv_I^-(\cL)$ is empty and the claim trivially follows. 
Otherwise, the points $x\in \inv_I^-(\cL)$ with $x_i\neq 0$ for $i\in I$ are necessarily Zariski-dense, 
and for a generic point $u\in \R^n$, the intersection points of $\inv_I^-(\cL)$ with $\cL^{\perp}+u$ belongs to $(\C^*)^I\times \C^{[n]\backslash I}$. 
Showing that these intersection points are real for generic $u$ implies it for all. 

Suppose that a point $a+\ii b$ belongs to the intersection of $\inv_I^-(\cL)$ with $\cL^{\perp}+u$
where $a,b\in \R^n$ and $a_j +\ii b_j \neq 0$ for every $j\in I$. 
Then $(a-u)+ \ii b$ belongs to $\cL^{\perp}$. Since $\cL^{\perp}$ is conjugation invariant, 
it follows that $b\in \cL^{\perp}$. In particular, $b^Tx=0$ for all $x\in \cL$.  
Since $a+\ii b$ belongs to $\inv_I^-(\cL)$, $\inv_I^-(a+\ii b)$ belongs to $\cL$. It follows that
$b^T  \inv_I^-(a+\ii b) =0$. Taking imaginary parts gives 
\[
0 \ \ = \ \ {\rm Im}\left( \sum_{j\in I} \frac{-b_j}{a_j + \ii b_j} \ + \  \sum_{j\not\in I}b_j(a_j + \ii b_j)\right) 
 \ \ = \ \  \sum_{j\in I} \frac{b_j^2}{a_j^2 +  b_j^2} \ +  \ \sum_{j\not\in I}b_j^2.
\]
Since every term is nonnegative and their sum is zero, each term must be zero. Thus $b_j=0$ for all $j$ 
and the point $a+\ii b$ is real.  
\end{proof}

\begin{rem}
Propostion~\ref{prop:realPoints} shows that $\inv_I^-(\cL)$ is \emph{hyperbolic} with respect to $\cL^{\perp}$, 
in the sense of \cite{SV}.  In fact, one can replace $\cL^{\perp}$ in this statement by any 
linear space of the same dimension whose non-zero Pl\"ucker coordinates agree in sign 
with those of $\cL^{\perp}$. This shows that $\inv_I^-(\cL)$ is a \emph{stable variety}.  See \cite[Section 2]{KV} for more. 
\end{rem}

\begin{prop} \label{prop:minF}
For generic $u\in \R^n$, the intersection points of  $\inv_I^-(\cL)$ with $\cL^{\perp} + u$ are the minima of the function 
\begin{equation}\label{eq:F}
f(x)  \ \  = \ \   \frac{1}{2}\sum_{j\not\in I} x_j^2 - \sum_{j\in I} \log|x_j|
\end{equation}
over the regions in the complement of the \textup{(}affine\textup{)} hyperplane arrangement $\{x_i=0\}_{i\in I}$ in the affine linear space $\cL^{\perp}+u$. 
\end{prop}

\begin{proof}
On $(\R^*)^I\times \R^{[n]\backslash I}$, $f$ is infinitely differentiable and 
we examine its behavior on each orthant. 
For a sign pattern $\sigma: I \rightarrow \{\pm 1\}$, let $\R^I_\sigma$ denote the orthant 
of points in $(\R^*)^I$ with $\sigma(i)x_i>0$ for all $i\in I$. 
 Inspecting the Hessian of $f$ shows that it is  also strictly convex on $\R_\sigma^I\times \R^{[n]\backslash I}$ . 
 Indeed, the Hessian of $f$ is a diagonal matrix 
whose $(j,j)$th entry is equal to $1/x_j^2$ for $j\in I$ and $1$ for $j\not\in I$ and is therefore positive 
definite on $(\R^*)^I \times \R^{[n]\backslash I}$. 

Define the (open) polyhedron $\cP_\sigma$ to be the intersection of $\R^I_\sigma\times \R^{[n]\backslash I}$ with 
the affine space $\cL^{\perp}+u$. 
The function $f$ is strictly convex on $\cP_\sigma$. Therefore any critical point of $f$ over $\cP_\sigma$ is a global minimum. 
The affine span of $ \cP_\sigma$ is $\cL^{\perp}+u$, so $p \in \cP_\sigma$ is a critical point of $f$ 
when $\nabla f (p)$ belongs to $(\cL^{\perp})^{\perp} = \cL$.  Since $\nabla f(p) = \inv_I^-(p)$ and $\inv_I^-$ is an involution, 
this implies that $p$ belongs to $\inv_I^-(\cL)$. 
Putting this all together, we find that for a point $p\in \cP_{\sigma}$,  
\[
p \text{ attains the minimum of $f$ over $\cP_{\sigma}$ } 
\ \Leftrightarrow \
\nabla f(p) \in \cL 
\ \Leftrightarrow \
p \in \inv_I^-(\cL). 
\]
\end{proof}

We can characterize which connected components of $(\cL^{\perp}+u)\backslash\{x_i=0\}_{i\in I}$ 
contains a point in $\inv_I^-(\cL)$ in terms of the recession cone ${\rm rec}(\cP_{\sigma}) = (\R^I_\sigma\times \R^{[n]\backslash I}) \cap \cL^{\perp}$.

\begin{lem} \label{lem:infF}
The infimum of $f$ over $\cP_\sigma$ is attained if and only if the intersection of $\R^I$ with the recession cone of $\cP_\sigma$ is trivial, i.e. 
${\rm rec}(\cP_\sigma) \cap \R^I = \{0\}$. 
\end{lem}

\begin{proof}
($\Rightarrow$)  Suppose ${\rm rec}(\cP_\sigma)\cap \R^I$ contains $v\neq 0$.  Then for any $p\in \cP_\sigma$,
the univariate function $f(p+tv) = \frac{1}{2}\sum_{j\not\in I}p_j^2 - \sum_{i\in I}\log|p_i + t v_i|$ is strictly decreasing 
as $t\rightarrow \infty$ and the infimum of $f$ is not attained on $\cP_{\sigma}$. 

($\Leftarrow$) 
Suppose that ${\rm rec}(\cP_\sigma) \cap \R^I = \{0\}$. Then the quadratic form $\sum_{j \not\in I}x_j^2$ is 
positive definite on the recession cone ${\rm rec}(\cP_\sigma)$. 
We can write $\cP_\sigma$ as $Q + {\rm rec}(\cP_{\sigma})$, where $Q$ is a compact polytope. 
Let $S$ denote the section of the recession cone, ${S =  \{v \in {\rm rec}(\cP_{\sigma}): || v||_1 =1\}.}$
For any point $p\in Q$ and $v\in S$, consider the univariate function $t\mapsto f(p+tv)$, which is strictly convex and continuous on 
$\{t : p + tv\in \cP_\sigma\}$. Its derivative 
\[
\frac{d}{dt}f(p+tv)  \ \ = \ \  
\sum_{j \not\in I}v_jp_j \ + \ t\sum_{j \not\in I}v_j^2  \ - \ \sum_{i\in I}\frac{v_i}{p_i+tv_i}
\] 
has a unique root $t\in \R$ for $p + tv\in \cP_\sigma$. 
 Indeed, by assumption $\sum_{j \not\in I}v_j^2>0$. Then, 
 since $\frac{d^2}{dt^2}f(p+tv) >0$ where defined, 
$\frac{d}{dt}f(p+tv)$ is strictly increasing on $\{t : p + tv\in \cP_\sigma\}$. 
If $v\in \R^{[n]\backslash I}$, then 
this set is all of $\R$ and $\frac{d}{dt}f(p+tv)$ is linear. Otherwise, there is a minimum 
$t$ for which  $p + tv\in \cP_\sigma$ and $\frac{d}{dt}f(p+tv)\rightarrow -\infty$ as $t$
approaches this minimum, whereas $ \frac{d}{dt}f(p+tv)>0$ for sufficiently large $t$. 
Let $t^*(p,v)$ denote this unique root of $\frac{d}{dt}f(p+tv)$.  This is a continuous function 
in $p$ and $v$.  Let $T$ denote the maximum of $t^*(p,v)$ over $(p,v) \in Q\times S$. 

Now we claim that when minimizing $f$ over $\cP_\sigma$, it suffices to 
minimize over the compact set $Q + [0,T]S$. Indeed, if $y\in  \cP_\sigma$, 
then $y = p + tv$ for some $p\in Q$, $v\in S$ and $t\in \R_{>0}$. 
If $t>T$, then the point $x =  p + Tv \in Q + [0,T]S$ satisfies $f(x)< f(y)$. 
In particular, the minimum of $f$ is bounded from below and is therefore attained on the compact 
set $Q + [0,T]S$. 
\end{proof}

\begin{prop}
For generic  $u\in \R^n$, there is exactly one point of $\inv_I^-(\cL)$ in each region of $(\cL^{\perp}+u)\backslash\{x_i=0\}_{i\in I}$ 
whose recession cone has trivial intersection with $\R^I$. 
The degree of $\inv_I^-(\cL)$ equals the number of these regions. 
\end{prop}

\begin{proof}
First we show that for generic $u\in \R^n$, the number of intersection points of $\inv_I^-(\cL)$ with $\cL^{\perp}+u$ equals the degree of $\inv_I^-(\cL)$. To do this, 
we show that the closure $\ol{\inv_I^-(\cL)}$ in $\PP^n(\C)$ has no points in common with $\cL^{\perp}+x_0u$ with $x_0 = 0$. 
For the sake of contradiction suppose that for some $a\in \cL^{\perp}$, the point $[0: a]$ belongs to $\ol{\inv_I^-(\cL)}$ and let $S = {\rm supp}(a)$.  

It follows that $a^T\x = \sum_{i\in S}a_i x_i$ vanishes on $\cL$, $g = \x^{S\cap I}\cdot  a^T\inv_I^-(\x)$ vanishes on $\inv_I^-(\cL)$, and the  
homogenezation $g^{\rm hom}$ with respect to $x_0$ vanishes on the closure $\ol{\inv_I^-(\cL)}\subseteq \PP^n(\C)$. 
In particular, $g^{\rm hom} (0,a) =0$.  If $S\subseteq I$,  this contradicts the evaluation of the polynomial $g^{\rm hom} = g = \sum_{j\in S}a_j \x^{S\backslash j}$ given by 
\[ g^{\rm hom}(0,a) = \sum_{j\in S}a^S = a^S\cdot |S| \neq 0.\]
Similarly, since $\ol{\inv_I^-(\cL)}$ is invariant under complex conjugation, we also have $g^{\rm hom} (0,\ol{a}) =0$, where $\ol{a}$ is the complex conjugate of $a$. 
If $S\not\subseteq I$, this contradicts the evaluation of the polynomial 
$g^{\rm hom} = -x_0^2\sum_{j\in S\cap I}a_j \x^{S\cap I \backslash j} + \x^{S\cap I}\sum_{j\in S\backslash I}a_j x_j$  given by  
\[ g^{\rm hom}(0,\ol{a})  =  \ol{a}^{S\cap I} \sum_{j\in S\backslash I}a_j \ol{a_j} \neq 0.\]
Therefore all the intersection points of $\ol{\inv_I^-(\cL)}$ with $\cL^{\perp}+x_0u$ have $x_0\neq 0$. 
Then for generic $u$, the number of intersection points of $\inv_I^-(\cL)$ and $\cL^{\perp}+u$ equals the degree of $\inv_I^-(\cL)$. 

By Propositions~\ref{prop:realPoints} and \ref{prop:minF}, each of these intersection points is real 
and thus is a minimizer of the function $f(x)$ of \eqref{eq:F} over some connected component $\cP_{\sigma}$ of $(\cL^{\perp}+u)\backslash\{x_i=0\}_{i\in I}$.  
By Lemma~\ref{lem:infF}, the components $\cP_{\sigma}$
contains a minimizer if and only if ${\rm rec}(\cP_{\sigma})\cap \R^{I}= \{0\}$.
\end{proof}

This together with Corollary~\ref{cor:degreeFacets} constitutes the proof of Theorem~\ref{thm:real}. 
For special cases of $I$, we find a simpler characterization of the regions counted by $\deg(\inv_I(\cL))$. 

\begin{cor} Let $u\in \R^n$ be generic. 
If $I$ is independent in the matroid $M(\cL)$, then the degree of $\inv_I(\cL)$ 
equals the total number of regions in ${(\cL^{\perp}+u)\backslash\{x_i=0\}_{i\in I}}$. 
If $I = [n]$, then the degree of $\inv_I(\cL)$ 
equals the number of bounded regions in ${(\cL^{\perp}+u)\backslash \{x_i=0\}_{i\in I}}$. 
\end{cor}

\begin{proof} 
If $I$ is independent in $M(\cL)$, then $I$ is contained in a basis $B$ of $M(\cL)$, and 
$[n]\backslash B$ is a basis of $M(\cL^{\perp})$ contained in $[n]\backslash I$.  
In particular, if $x\in \cL^{\perp}$ has $x_j = 0$ for all $j\in [n]\backslash I$, then $x =0$. 
So $\R^I\cap  \cL^{\perp} = \{0\}$.  The recession cone of any region of $(\cL^{\perp}+u)\backslash\{x_i=0\}_{i\in I}$
is contained in $\cL^{\perp}$, so its intersection with $\R^I$ is trivial. 

If $I = [n]$, then $\R^I = \R^n$. The recession cone of a region in $(\cL^{\perp}+u)\backslash \{x_i=0\}_{i\in I}$
contains a non-zero vector if and only if it is unbounded. Therefore the regions whose 
recession cones have trivial intersection with $\R^I$ are those which are bounded. 
\end{proof}

\begin{ex}\label{ex:NonGen2}
Consider the $3$-dimensional linear space $\cL$ from Example~\ref{ex:NonGen} and take the vector  
$u =(0,0,1,2,2)$. The two-dimensional affine space $\cL^{\perp}+ u$ consists of points 
of the form 
$(x_1, x_2,  x_1 - x_2+1, - x_2+2, - x_1 + x_2+2)$. 
Since $I = \{1,2,3\}$ is independent in $M(\cL)$, each of the seven regions in the complement of the hyperplane arrangement 
$\{x_i=0\}_{i\in I}$ contains a point of $\inv_I^-(\cL)$.
For $I = \{1,2,3,4\}$, there are four regions whose recession cones intersect $\{x_5=0\}$ nontrivially. 
The remaining six regions each contain a unique point in $\inv_I^-(\cL)$. 
Finally, for $I = \{1,2,3,4,5\}$, $\R^I$ is all of $\R^5$ so the recession cone of $\cP_{\sigma}$ intersects $\R^I$ 
nontrivially if and only if $\cP_{\sigma}$ is unbounded. Thus the four bounded regions 
of the hyperplane arrangement $\{x_i = 0\}_{i\in I}$ in $\cL^{\perp} + u$ are precisely those that 
contain points in $\inv_I^-(\cL)$.  These hyperplane arrangements and intersection points 
are shown in Figure~\ref{fig:Hyp}.    
\end{ex}

\begin{center}
\begin{figure}
\includegraphics[width=1.6in]{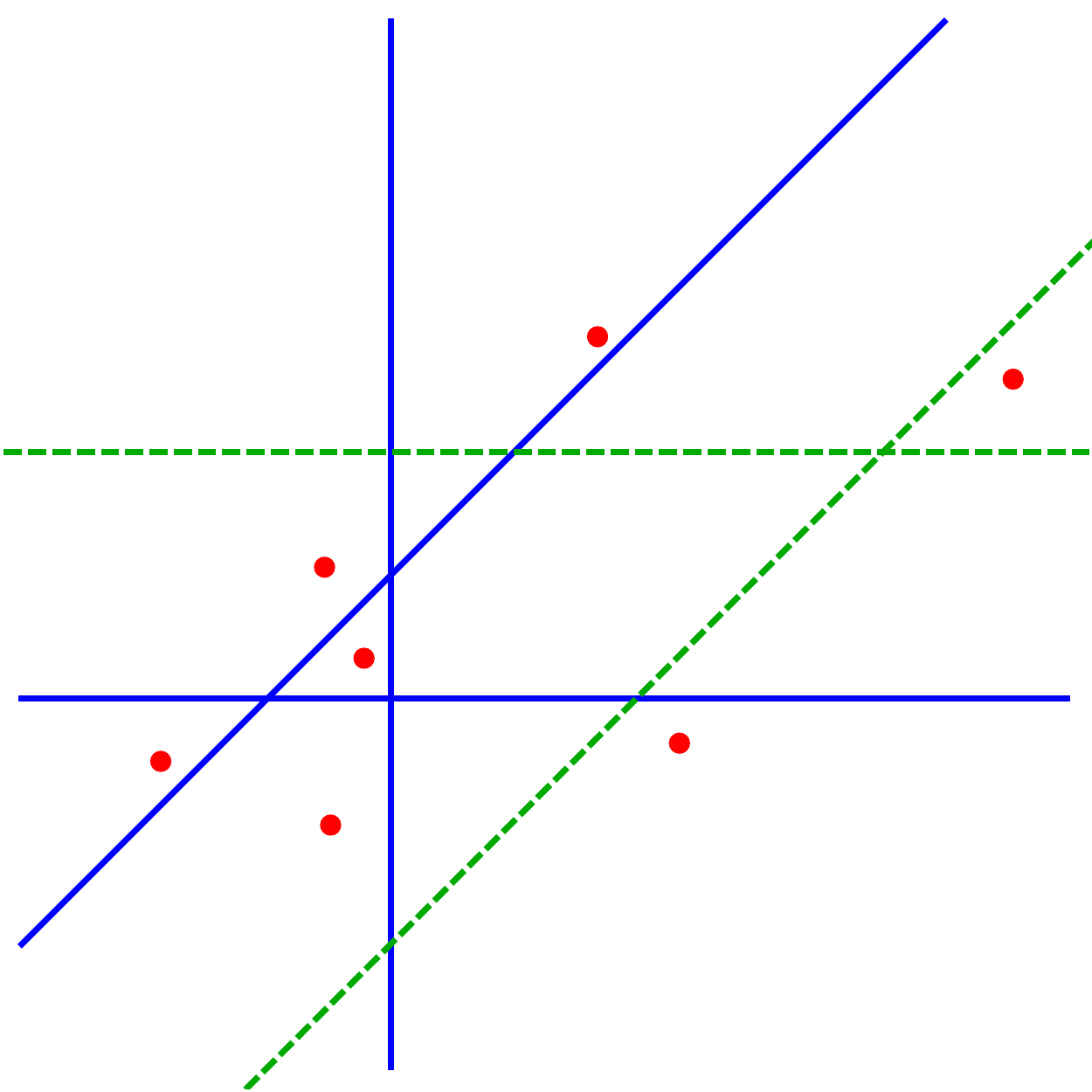} \hfill
\includegraphics[width=1.6in]{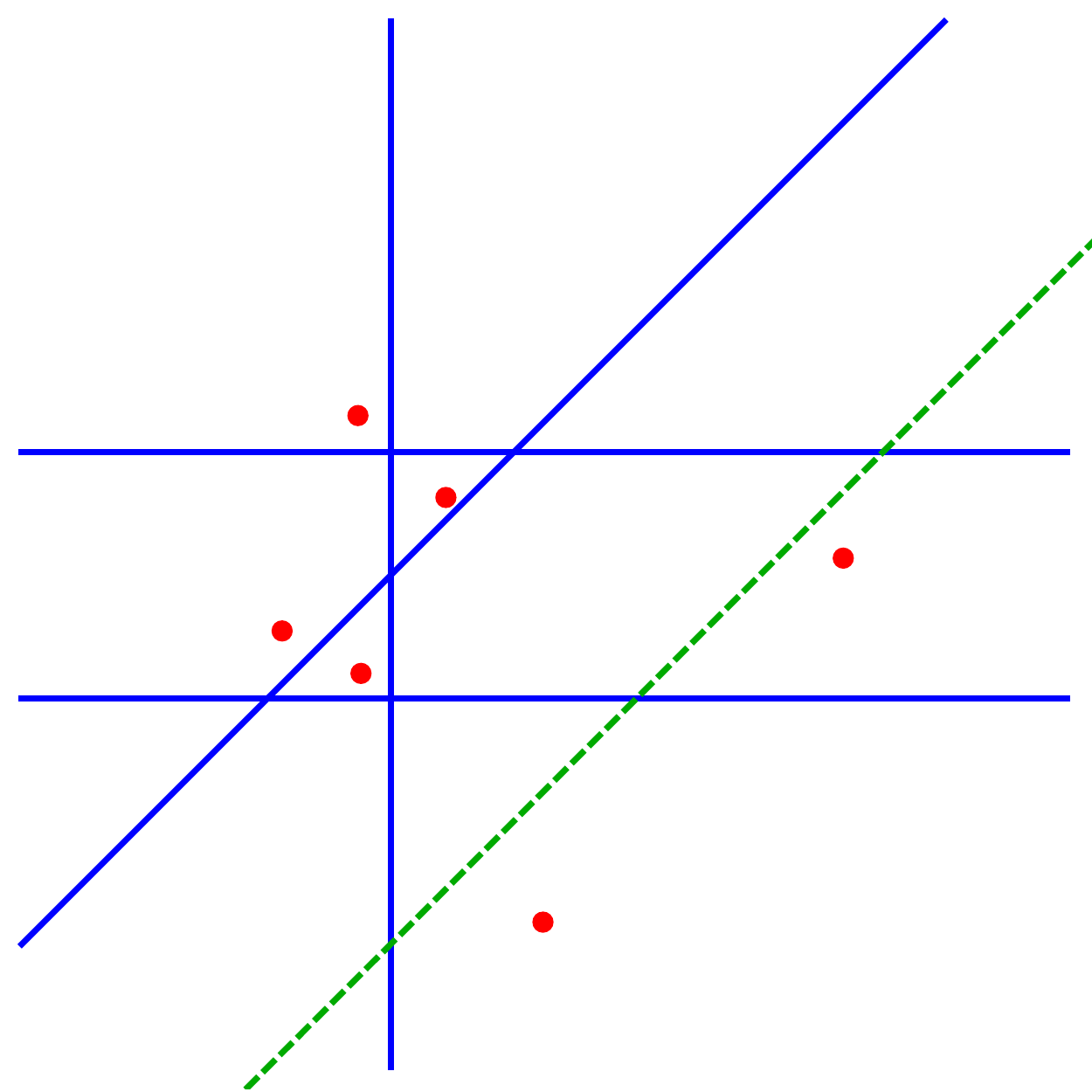} \hfill
\includegraphics[width=1.6in]{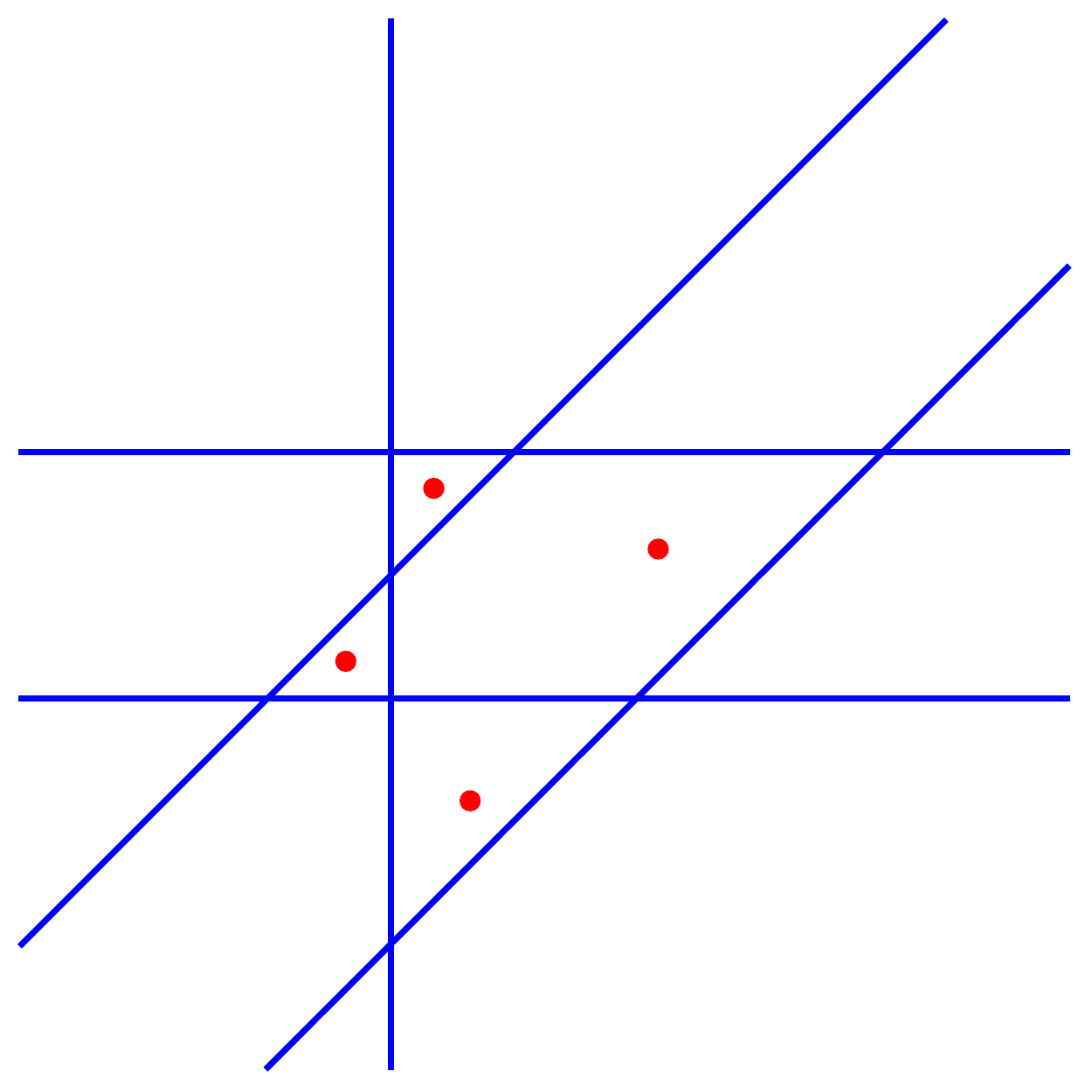}
\caption{Intersections of $\cL^{\perp} + u$ with $\inv_I^-(\cL)$ from Example~\ref{ex:NonGen2}.}
\label{fig:Hyp}
\end{figure}
\end{center}


\begin{thebibliography}{10}

\bibitem{AB16}
F.~Ardila and A.~Boocher.
\newblock The closure of a linear space in a product of lines.
\newblock {\em J. Algebraic Combin.}, 43(1):199--235, 2016.

\bibitem{ACS16}
F.~Ardila, F.~Castillo, and J.~A. Samper.
\newblock The topology of the external activity complex of a matroid.
\newblock {\em Electron. J. Combin.}, 23(3):Paper 3.8, 20, 2016.

\bibitem{BW97}
A.~Bj\"{o}rner and M.~L. Wachs.
\newblock Shellable nonpure complexes and posets. {II}.
\newblock {\em Trans. Amer. Math. Soc.}, 349(10):3945--3975, 1997.

\bibitem{DLSV12}
J.~A. De~Loera, B.~Sturmfels, and C.~Vinzant.
\newblock The central curve in linear programming.
\newblock {\em Found. Comput. Math.}, 12(4):509--540, 2012.

\bibitem{FSW17}
A.~Fink, D.~E. Speyer, and A.~Woo.
\newblock A {G}r\"obner basis for the graph of the reciprocal plane.
\newblock Preprint, available at \url{http://arxiv.org/abs/1703.05967}, 2017.

\bibitem{HuhWang}
J.~Huh and B.~Wang.
\newblock Enumeration of points, lines, planes, etc.
\newblock {\em Acta Math.}, 218(2):297--317, 2017.

\bibitem{KV}
M.~Kummer and C.~Vinzant.
\newblock The {C}how form of a reciprocal linear space.
\newblock Preprint, available at \url{http://arxiv.org/abs/1610.04584}, 2016.

\bibitem{tropBook}
D.~Maclagan and B.~Sturmfels.
\newblock {\em Introduction to {T}ropical {G}eometry}, volume 161 of {\em Graduate
  Studies in Mathematics}.
\newblock American Mathematical Society, Providence, RI, 2015.

\bibitem{MSUZ14}
M.~Micha{\l}ek, B.~Sturmfels, C.~Uhler, and P.~Zwiernik.
\newblock Exponential varieties.
\newblock {\em Proc. Lond. Math. Soc. (3)}, 112(1):27--56, 2016.

\bibitem{CCABook}
E.~Miller and B.~Sturmfels.
\newblock {\em Combinatorial commutative algebra}, volume 227 of {\em Graduate
  Texts in Mathematics}.
\newblock Springer-Verlag, New York, 2005.

\bibitem{Oxley}
J.~Oxley.
\newblock {\em Matroid theory}, volume~21 of {\em Oxford Graduate Texts in
  Mathematics}.
\newblock Oxford University Press, Oxford, second edition, 2011.

\bibitem{PS}
N.~Proudfoot and D.~Speyer.
\newblock A broken circuit ring.
\newblock {\em Beitr\"age Algebra Geom.}, 47(1):161--166, 2006.

\bibitem{ProudfootXuYoung}
N.~Proudfoot, Y.~Xu, and B.~Young.
\newblock The {$Z$}-polynomial of a matroid.
\newblock {\em Electron. J. Combin.}, 25(1):Paper 1.26, 21, 2018.

\bibitem{SSV13}
R.~Sanyal, B.~Sturmfels, and C.~Vinzant.
\newblock The entropic discriminant.
\newblock {\em Adv. Math.}, 244:678--707, 2013.

\bibitem{SV}
E.~Shamovich and V.~Vinnikov.
\newblock Livsic-type determinantal representations and hyperbolicity.
\newblock {\em Advances in Mathematics}, 329:487 -- 522, 2018.

\bibitem{StanleyCCABook}
R.~P. Stanley.
\newblock {\em Combinatorics and commutative algebra}, volume~41 of {\em
  Progress in Mathematics}.
\newblock Birkh\"{a}user Boston, Inc., Boston, MA, second edition, 1996.

\bibitem{GrobnerPolytope}
B.~Sturmfels.
\newblock {\em Gr\"obner bases and convex polytopes}, volume~8 of {\em
  University Lecture Series}.
\newblock American Mathematical Society, Providence, RI, 1996.

\bibitem{Terao}
H.~Terao.
\newblock Algebras generated by reciprocals of linear forms.
\newblock {\em J. Algebra}, 250(2):549--558, 2002.

\bibitem{Var95}
A.~Varchenko.
\newblock Critical points of the product of powers of linear functions and
  families of bases of singular vectors.
\newblock {\em Compositio Math.}, 97(3):385--401, 1995.

\bibitem{wagnerSurvey}
D.~G. Wagner.
\newblock Multivariate stable polynomials: theory and applications.
\newblock {\em Bull. Amer. Math. Soc. (N.S.)}, 48(1):53--84, 2011.

\end{thebibliography}
%
%
%
%
%

\end{document}